\documentclass[reqno,11pt]{amsart}

\usepackage[utf8]{inputenx} 
\usepackage[T1]{fontenc}    
\usepackage[margin=3cm]{geometry}

\usepackage{amssymb}   
\usepackage{amsthm}    
\usepackage{thmtools}  
\usepackage{mathtools} 
\usepackage{mathrsfs}  
\usepackage{commath}

\usepackage[all]{xy}
\usepackage{todonotes}
\allowdisplaybreaks

\RequirePackage[backend    = biber,
                sortcites  = true,
                giveninits = true,
                doi        = false,
                isbn       = false,
                url        = false,
                style      = numeric-comp,
                maxbibnames=99]{biblatex}
\DeclareNameAlias{sortname}{family-given}
\DeclareNameAlias{default}{family-given}
\usepackage{varioref}
\usepackage{hyperref}
\usepackage[nameinlink, capitalize, noabbrev]{cleveref}

\declaretheorem[style = plain, numberwithin = section]{theorem}
\declaretheorem[style = plain,      sibling = theorem]{corollary}
\declaretheorem[style = plain,      sibling = theorem]{lemma}
\declaretheorem[style = plain,      sibling = theorem]{proposition}

\declaretheorem[style = definition, sibling = theorem]{problem}

\DeclareMathOperator{\vol}{vol}
\DeclareMathOperator{\spn}{span}
\newcommand{\clspn}{\overline{\spn}}

\newcommand{\B}{\mathcal{B}}
\DeclareMathOperator{\Tr}{Tr}

\newcommand{\N}{\mathbb{N}}   
\newcommand{\R}{\mathbb{R}}   
\newcommand{\C}{\mathbb{C}}   
\newcommand{\T}{\mathbb{T}}   

\DeclareMathOperator{\tr}{tr}

\DeclareMathOperator{\id}{id}
\DeclareMathOperator{\vN}{L}
\DeclareMathOperator{\rvN}{R}
\DeclareMathOperator{\cdim}{cdim}
\newcommand{\Hi}{\mathcal{H}}
\newcommand{\Hip}{\mathcal{H}_{\pi}}
\DeclareMathOperator{\Center}{\mathcal{Z}}

\title[The density theorem for projective representations]{The density theorem for projective representations via twisted group von Neumann algebras}
\author{Ulrik Enstad}
\address{Department of Mathematics, University of Oslo, 0851 Oslo, Norway.}
\email{ubenstad@math.uio.no}


\begin{document}

\maketitle

\begin{abstract}
    We consider converses to the density theorem for irreducible, projective, unitary group representations restricted to lattices using the dimension theory of Hilbert modules over twisted group von Neumann algebras. We show that under the right assumptions, the restriction of a $\sigma$-projective unitary representation $\pi$ of a group $G$ to a lattice $\Gamma$ extends to a Hilbert module over the twisted group von Neumann algebra $\vN(\Gamma,\sigma)$. We then compute the center-valued von Neumann dimension of this Hilbert module. For abelian groups with 2-cocycle satisfying Kleppner's condition, we show that the center-valued von Neumann dimension reduces to the scalar value $d_{\pi} \vol(G/\Gamma)$, where $d_{\pi}$ is the formal dimension of $\pi$ and $\vol(G/\Gamma)$ is the covolume of $\Gamma$ in $G$. We apply our results to characterize the existence of multiwindow super frames and Riesz sequences associated to $\pi$ and $\Gamma$. In particular, we characterize when a lattice in the time-frequency plane of a second countable, locally compact abelian group admits a Gabor frame or Gabor Riesz sequence.
\end{abstract}

\section{Introduction}

Let $(\pi, \Hip)$ be a unitary representation of a locally compact group $G$ and let $\Gamma$ be a discrete subset of $G$. The study of the spanning and linear independence properties of $\Gamma$-indexed families of the form
\begin{equation}
     \pi(\Gamma) \eta = ( \pi(\gamma) \eta )_{\gamma \in \Gamma} \label{eq:orbits}
\end{equation}
for $\eta \in \Hip$ is fundamental in many areas of applied harmonic analysis, including time-frequency analysis and wavelet theory \cite{Gr01,HaLa00, Fu05, DaGrMe86, GrMoPa85}. Under certain assumptions on $G$ and $\pi$, fundamental results known as \emph{density theorems} provide basic obstructions to the spanning and linear independence properties of such families depending only on notions of density of $\Gamma$ in $G$ \cite{FuGrHa17,Ku07}.

While there exist many different spanning properties for families in Hilbert spaces, we will focus on frames in the present paper. Unlike e.g.\ Schauder bases, frames provide unconditionally convergent expansions of every element in the Hilbert space, making them ideal for applied harmonic analysis. The dual notion is that of a Riesz sequence, which is a strong notion of linear independence (see \Cref{subsec:frames_riesz_sequences} for definitions).

When $\Gamma$ is a lattice in $G$, i.e.\ a discrete subgroup with finite covolume $\vol(G/\Gamma)$, the uniform density of $\Gamma$ is given by $1/\vol(G/\Gamma)$. In \cite{VaRo00}, Romero and van Velthoven proves the following general density theorem for frames and Riesz sequences of the form \eqref{eq:orbits}:

\begin{theorem}\label{thm:density_theorem}
Let $(\pi,\Hip)$ be an irreducible, square-integrable, unitary representation of a second-countable, unimodular, locally compact group $G$ with formal dimension $d_{\pi}$, and let $\Gamma$ be a lattice in $G$. Then the following hold for $\eta \in \Hip$:
\begin{enumerate}
    \item If $\pi(\Gamma)\eta$ is a frame for $\Hip$, then $d_{\pi} \vol(G/\Gamma) \leq 1$.
    \item If $\pi(\Gamma) \eta$ is a Riesz sequence for $\Hip$, then $d_{\pi} \vol(G/\Gamma) \geq 1$.
\end{enumerate}
\end{theorem}

In fact, Romero and van Velthoven prove \Cref{thm:density_theorem} more generally for projective unitary representations \cite[Theorem 7.4]{VaRo00}, that is, maps $\pi$ from $G$ into the unitary operators on $\Hip$ that satisfy a composition rule of the form
\[ \pi(x)\pi(y) = \sigma(x,y) \pi(xy) \;\;\; \text{for all $x,y \in G$}. \]
Here, $\sigma$ is an associated measurable function $G \times G \to \T$ called a 2-cocycle, and ordinary representations correspond to $\sigma=1$.

We will work with projective representations since one of our motivating examples is projective, namely the Weyl--Heisenberg representation associated to a locally compact abelian group $A$. For $A = \R^d$, this is the projective representation of $G = \R^d \times \R^d \cong\ \R^{2d}$ on $L^2(\R^d)$ given by
\[ \pi(x,\omega) \xi(t) = e^{2\pi i \omega t} \xi(t-x) \;\;\; \text{for $(x,\omega) \in \R^d \times \R^d$ and $\xi \in L^2(\R^d)$.} \]
In this context a system of the form $\pi(\Gamma) \eta$ for some lattice $\Gamma$ in $\R^{2d}$ and $\eta \in L^2(\R^d)$ is known as a Gabor system, and the study of their spanning and linear independence properties forms the basis of Gabor analysis. The density theorem has a long history in Gabor analysis, see \cite{He07, Ba90, Da90, RaSt95, JaLe16}. We delay the definition of the Weyl--Heisenberg representation in the setting of a locally compact abelian group until \Cref{subsec:gabor}.

In the present paper we consider converses to the lattice density theorem for projective representations, that is, the following problem:

\begin{problem}\label{problem:1}
Let $(\pi,\Hip)$ be a projective, irreducible, square-integrable, unitary representation of a second-countable, unimodular, locally compact group $G$ with formal dimension $d_{\pi} > 0$, and let $\Gamma$ be a lattice in $G$.
\begin{enumerate}
    \item Does $d_{\pi} \vol(G/\Gamma) \leq 1$ imply the existence of $\eta \in \Hip$ such that $\pi(\Gamma) \eta$ is a frame for $\Hip$?
    \item Does $d_{\pi} \vol(G/\Gamma) \geq 1$ imply the existence of $\eta \in \Hip$ such that $\pi(\Gamma) \eta$ is a Riesz sequence for $\Hip$?
\end{enumerate}
\end{problem}

The above problem has been considered several times in the literature. For Gabor frames it was an open problem for several years with important partial progress made by Daubechies, Grossmann, and Morlet in the one-dimensional case $d=1$ \cite{DaGrMe86} and by Han and Wang for separable lattices in higher dimensions \cite{HaWa01}. Bekka settled the problem in \cite{Be04} where it is proved that the condition $\vol(G/\Gamma) \leq 1$ is sufficient for the existence of a Gabor frame $\pi(\Gamma)\eta$ for arbitrary lattices $\Gamma$ in $\R^{2d}$. We mention that the result can also be inferred from a computation of Rieffel \cite[Theorem 3.5]{Ri88}.

A natural approach to the density theorem and its possible converses goes via the dimension theory for Hilbert modules over von Neumann algebras, as demonstrated by Bekka. In \cite{Be04}, Bekka works in the general setting of a (nonprojective) irreducible, square-integrable, unitary representation $(\pi,\Hip)$ of a second-countable, unimodular, locally compact group $G$. It is shown that the existence of a frame $\pi(\Gamma) \eta$ for some $\eta \in \Hip$ and lattice $\Gamma \subseteq G$ is equivalent to $\cdim_{\vN(\Gamma)} \Hip \leq I$, where $\cdim_{\vN(\Gamma)} \Hip$ is the center-valued von Neumann dimension of $\Hip$ as a Hilbert module over the group von Neumann algebra of $\Gamma$. Thus, in situations where $\cdim_{\vN(\Gamma)} \Hip$ collapses to the scalar operator $d_{\pi} \vol(G/\Gamma) I$, the first part of \Cref{problem:1} has a positive answer. However, this is not the case e.g.\ for representations in the discrete series of $SL(2,\R)$ \cite[Example 1]{Be04}. When $\vN(\Gamma)$ is a factor, i.e.\ when $\Gamma$ is an ICC group, the expression for $\cdim_{\vN(\Gamma)} \Hip$ reduces to the famous formula
\[ \dim_{\vN(\Gamma)} \Hip = d_{\pi} \vol(G/\Gamma) \]
which goes back to the work of Atiyah--Schmid \cite{AtSc77}; see also \cite[Theorem 3.1.1]{GodeJo89}.

In \cite{VaRo00}, Romero and van Velthoven also considers the converse of \Cref{thm:density_theorem} for general projective representations. In particular, they show that when $(\Gamma,\sigma)$ satisfies Kleppner's condition (see \Cref{subsec:cocycles}), then $d_{\pi} \vol(G/\Gamma) \leq 1$ (resp.\ $d_{\pi} \vol(G/\Gamma) \geq 1)$ implies the existence of a frame (resp.\ Riesz sequence) of the form $\pi(\Gamma)\eta$ for some $\eta \in \Hip$. In the present paper we approach their result from the viewpoint of twisted group von Neumann algebras by adapting Bekka's arguments to the setting of projective representations. Thus, the group von Neumann algebra $\vN(\Gamma)$ is replaced by the twisted group von Neumann algebra $\vN(\Gamma,\sigma)$. Twisted group operator algebras were introduced by Zeller-Meier in \cite{ZM68} and have been studied in e.g.\ \cite{Om14,BeCo09,Kl62,Ri88,Pa89,BeOm18}. We state two of our main results in the following theorem, which are adaptions of Bekka's result \cite[Theorem 1]{Be04}:

\begin{theorem}\label{thm:intro1}
Let $(\pi,\Hip)$ be a $\sigma$-projective, irreducible, square-integrable, unitary representation of a second-countable, unimodular, locally compact group $G$ and let $\Gamma$ be a lattice in $G$. Then the representation $\pi|_{\Gamma}$ of $\Gamma$ extends to give $\Hip$ the structure of a Hilbert $\vN(\Gamma,\sigma)$-module. The center-valued von Neumann dimension of $\Hip$ is the (possibly unbounded) operator on $\ell^2(\Gamma)$ given by $\sigma$-twisted convolution $f \mapsto \phi *_{\sigma} f$ with the function $\phi$ on $\Gamma$ defined as follows:
\[ \phi(\gamma) = \begin{cases} \displaystyle  \frac{d_{\pi}}{|C_{\gamma}|} \int_{G/\Gamma_{\gamma}} \overline{\sigma(\gamma,y)}\sigma(y,y^{-1} \gamma y) \langle \eta, \pi(y^{-1} \gamma y) \eta \rangle \dif{(y \Gamma_{\gamma})}   & \text{if $C_{\gamma}$ is $\sigma$-regular and finite,} \\ 0 & \text{otherwise.} \end{cases} \]
Here, $C_{\gamma}$ denotes the conjugacy class of $\gamma$, $\Gamma_{\gamma}$ denotes the centralizer of $\gamma \in \Gamma$, and $\eta$ is any unit vector in $\Hip$.

Moreover, the following hold:
\begin{enumerate}
    \item There exists a frame of the form $\pi(\Gamma) \eta$ for some $\eta \in \Hip$ if and only if $\delta_e - \phi$ is a $\sigma$-positive definite function.
    \item There exists a Riesz sequence of the form $\pi(\Gamma) \eta$ for some $\eta \in \Hip$ if and only if $\phi - \delta_e$ is a $\sigma$-positive definite function.
\end{enumerate}
\end{theorem}

When $\sigma$ is the trivial 2-cocycle, \cite[Theorem 1]{Be04} is recovered. A special case of \cref{thm:intro1} occurs when $(\Gamma,\sigma)$ satisfies Kleppner's condition. This condition equivalent to the factoriality of $\vN(\Gamma,\sigma)$, in which case the center-valued von Neumann dimension of $\Hip$ reduces to the scalar operator $d_{\pi} \vol(G/\Gamma) I$. Thus, we recover the result of Romero and van Velthoven \cite[Theorem 1.1]{VaRo00}. The scalar-valued von Neumann dimension in the projective setting was also computed for certain representations in \cite{Ra98}.

The description of $\cdim_{\vN(\Gamma,\sigma)} \Hip$ in \Cref{thm:intro1} is obtained along the same lines as in the untwisted case in \cite{Be04} with the necessary modifications needed to incorporate the 2-cocycle $\sigma$; see \Cref{sec:dimension}. An important ingredient in the proof is a description of the center-valued trace for twisted group von Neumann algebras, which does not seem to have appeared in the literature before; see \Cref{thm:cvt}. We also show in \Cref{cor:abelian_cdim} that the situation is particularly simple when $G$ is abelian and the whole group $G$ satisfies Kleppner's condition with respect to $\sigma$: In this case, the center-valued von Neumann dimension collapses to the scalar operator $d_{\pi} \vol(G/\Gamma) I$. Thus we get a complete converse to the density theorem in this setting, which we state here as a theorem:

\begin{theorem}\label{thm:intro2}
Let $(\pi,\Hip)$ be a $\sigma$-projective representation of a locally compact group satisfying the hypotheses of \Cref{thm:intro1}. Assume additionally that $G$ is abelian and that $(G,\sigma)$ satisfies Kleppner's condition. Then for any lattice $\Gamma$ in $G$, we have that
\[ \cdim_{\vN(\Gamma,\sigma)} \Hip = d_{\pi} \vol(G/\Gamma) I . \]
Consequently, the following hold:
\begin{enumerate}
    \item There exists a frame of the form $\pi(\Gamma) \eta$ for some $\eta \in \Hip$ if and only if $d_{\pi} \vol(G/\Gamma) \leq 1$.
    \item There exists a Riesz sequence of the form $\pi(\Gamma) \eta$ for some $\eta \in \Hip$ if and only if $d_{\pi} \vol(G/\Gamma) \geq 1$.
\end{enumerate}
\end{theorem}

Note the difference between the Kleppner's condition of $(G,\sigma)$ and $(\Gamma,\sigma)$ in \Cref{thm:intro2}. In particular, the above result can be applied immediately to characterize the existence of Gabor frames and Gabor Riesz sequences over arbitrary lattices in the general locally compact abelian setting; see \Cref{thm:gabor_existence}. We remark that the argument given by Bekka in \cite{Be04} to characterize the existence of Gabor frames over lattices in $\R^d \times \R^d$ relies on the fact that Heisenberg group is a nilpotent Lie group and thus cannot be applied to Gabor frames in the general locally compact abelian setting, where no such structure is present.

We apply our results to characterize not only the existence of frames and Riesz sequences of the form $\pi(\Gamma) \eta$, but also more general multiwindow super systems associated to $\pi$. For wavelet and Gabor systems, multiwindow and super systems were introduced and studied systematically by Balan \cite{Ba98, Ba98-2} and Han and Larson \cite{HaLa00}. Density theorems for multiwindow and super systems in the Gabor case and beyond have been considered in \cite{AuJaLu20, BaDuHa20, GrLy09,JaLu18}. See \Cref{subsec:multi_super} for details.

Finally, we mention the related paper \cite{Ha17} which considers density theorems in the more general setting of a projective representation of a countable group, and the recent preprint \cite{AbSp20} which connects affine density with von Neumann dimension.

\subsection{Structure of paper}

The paper is structured as follows: In \Cref{sec:von_neumann} we cover the necessary background on Hilbert modules over finite von Neumann algebras and their scalar-valued and center-valued dimensions. In \Cref{sec:twisted_group} we introduce twisted group von Neumann algebras of discrete groups and describe their canonical center-valued trace. In \Cref{sec:dimension} we compute the center-valued von Neumann dimension of $\Hip$ as a Hilbert $\vN(\Gamma,\sigma)$-module. in \Cref{sec:frame_theory} we apply the results to frame theory, proving density theorems for multiwindow super systems along with converses.

\subsection{Acknowledgements}

The author would like to thank Sven Raum for helpful discussions, and Floris Elzinga and Franz Luef for giving valuable feedback on a draft of the paper.


\section{Hilbert modules over von Neumann algebras}\label{sec:von_neumann}

\subsection{Center-valued traces}

Let $M$ be a finite von Neumann algebra equipped with a faithful normal tracial state $\tau$. Denote by $L^2(M,\tau)$ the Hilbert space obtained from the GNS construction of $(M,\tau)$, and by $\Omega$ its cyclic vector. We will represent $M$ as operators on $L^2(M,\tau)$ unless otherwise stated.

Since $M$ is finite, it has a canonical center-valued trace, i.e.\ a normal bounded linear map $\Tr \colon M \to \Center(M)$ uniquely determined by the following properties:
\begin{enumerate}
    \item $\Tr(ab) = \Tr(ba)$ for all $a,b \in M$.
    \item $\Tr(ba) = b \Tr(a)$ for all $a \in M$ and $b \in \mathcal{Z}(M)$.
    \item $\Tr(a) = a$ for all $a \in \mathcal{Z}(M)$.
    \item $\Tr(a^*a) = 0$ implies $a=0$ for all $a \in M$.
\end{enumerate}
Moreover the center-valued trace and $\tau$ relates in the following way \cite[p.\ 278, Proposition 3]{Di77}:
\begin{equation}
    \tau(ab) = \tau(\Tr(a) b) = \tau(a \Tr(b) ) = \tau(\Tr(ab)) = \tau(\Tr(a)\Tr(b)). \label{eq:cvt_t}
\end{equation}

If $\mathcal{K}$ is a separable, infinite-dimensional Hilbert space with orthonormal basis $(e_i)_{i=1}^\infty$ then the von Neumann algebra $M \otimes \mathcal{B}(\mathcal{K})$ is semifinite. By \cite[p.\ 330, Theorem 2.34]{Ta02} it admits a faithful, semifinite, normal extended center-valued trace $\Phi$ which we can describe as follows: If $a$ is a positive element of $M \otimes \mathcal{B}(\mathcal{K})$ then $a$ can be decomposed into a matrix $(a_{ij})_{i,j=1}^\infty$ with entries in $M$ such that
\begin{equation*}
    \langle a (\xi \otimes e_i), \eta \otimes e_j \rangle = \langle a_{ij} \xi, \eta \rangle \;\;\; \text{for $a \in M \otimes \B(\mathcal{K})$, $\xi,\eta \in L^2(M,\tau)$ and $i,j \in \N$.} \label{eq:tensor_decomp}
\end{equation*}
Then $\Phi$ is defined as
\begin{equation}
    \Phi(a) = \sum_{i} \Tr(a_{ii}) , \label{eq:semi_finite}
\end{equation}
see \cite[p.\ 332]{Be04}. We also obtain a faithful, semifinite, normal extended scalar-valued trace $\tau \otimes \tr$ on $M \otimes \mathcal{B}(\mathcal{K})$ given by
\begin{equation}
    (\tau \otimes \tr)(a) = \sum_i \tau(a_{ii}) . \label{eq:semi_finite_scalar}
\end{equation}
The main property of $\Phi$ that we will need is the following proposition. We call projections $p$ and $q$ called \emph{Murray--von Neumann equivalent}, written $p \sim q$, if there exists a partial isometry $u \in M$ such that $u^*u = p$ and $uu^* = q$. We also write $p \precsim q$ when $p$ is Murray--von Neumann equivalent to a subprojection of $q$.

\begin{proposition}\label{prop:projections_semifinite}
Let $\Phi$ be a normal, faithful, semifinite extended center-valued trace on a von Neumann algebra $M$. Let $p$ and $q$ be projections in $M$. Then the following hold:
\begin{enumerate}
    \item If $p \precsim q$ then $\Phi(p) \leq \Phi(q)$, and if $p \sim q$ then $\Phi(p) = \Phi(q)$.
    \item Suppose that $p$ and $q$ are finite projections. Then $p \precsim q$ if and only if $\Phi(p) \leq \Phi(q)$, and $p \sim q$ if and only if $\Phi(p) = \Phi(q)$.
\end{enumerate}
\end{proposition}

\begin{proof}
The proof is essentially in \cite[Proposition 2]{Be04}, but we give it here for completeness. Part (1) goes as follows: If $p \precsim q$ then we can find a partial isometry $u \in M$ such that $p = u^*u$ and $uu^* \leq q$. By positivity and the tracial property of $\Phi$ this implies $\Phi(p) \leq \Phi(q)$. The implication from $p \sim q$ to $\Phi(p)=\Phi(q)$ now follows from anti-symmetry.

Before we prove part (2), we make the following observation: If $p',q'$ are finite projections in $M$ with $p' \precsim q'$ and $\Phi(p') = \Phi(q')$ then $p' \leq q'$. Indeed, letting $p' \sim p'' \leq q'$ then
\[ \Phi(p') = \Phi(q') = \Phi(p'') + \Phi(q'-p'') = \Phi(p') + \Phi(q'-p'') .\]
Since $p'$ is finite, \cite[p.\ 331, Proposition 2.35]{Ta02} implies that $\Phi(p')$ is finite almost everywhere as a function on the measure space $(X,\mu)$ such that $\mathcal{Z}(M) \cong L^\infty(X,\mu)$. Hence we may cancel $\Phi(p')$ in the above equation. The faithfulness of $\Phi$ now implies $q' = p''$, hence $p' \leq q'$.

We now prove (2). Suppose $p$ and $q$ are finite and that $\Phi(p) \leq \Phi(q)$. By \cite[p.\ 293, Theorem 1.8]{Ta02} we can find a projection $z$ in the center of $M$ such that $z p \precsim z q$ and $(1-z) q \precsim (1-z) p$. By part (1) we get $\Phi((1-z) q) \leq \Phi((1-z)p)$. However we also have that
\[ \Phi((1-z)p) = (1-z) \Phi(p) \leq (1-z) \Phi(q) = \Phi((1-z)q) . \]
By our observation applied to $p' = (1-z)q$ and $q' = (1-z)p$ we can conclude that $(1-z) q \leq (1-z) p$. Hence
\[ p = zp + (1-z)p \leq zp + (1-z)q \precsim zq + (1-z)q = q. \]
The implication from $\Phi(p) = \Phi(q)$ to $p \sim q$ follows from \cite[p.\ 291, Proposition 1.3]{Ta02}.
\end{proof}

\subsection{Hilbert modules}

We continue to assume that $M$ is a finite von Neumann algebra equipped with a faithful normal tracial state $\tau$. A \emph{(left) Hilbert $M$-module} is a Hilbert space $\Hi$ together with a unital normal $*$-homomorphism $\pi \colon M \to \mathcal{B}(H)$. We write $a \xi = \pi(a) \xi$ for $a \in M$ and $\xi \in \Hi$. We often write $_M \Hi$ when we want to emphasize that $\Hi$ is a Hilbert module over $M$. A Hilbert module $\Hi$ is called \emph{separable} if $\Hi$ is separable as a Hilbert space. Two Hilbert $M$-modules are \emph{isomorphic} if there exists an $M$-linear unitary operator between them.

Fixing a separable infinite-dimensional Hilbert space $\mathcal{K}$, we can turn the tensor product $L^2(M,\tau) \otimes \mathcal{K}$ into a Hilbert $M$-module via
\[ a (\xi \otimes \eta) = (a \xi) \otimes \eta \;\;\; \text{for $a \in M$, $\xi \in L^2(M,\tau)$ and $\eta \in \mathcal{K}$. } \]
This is nothing but the countable direct sum $\bigoplus_{j=1}^{\infty} L^2(M,\tau)$ where $M$ acts diagonally. One of the basic facts about Hilbert modules over $M$ is that under certain separability assumptions they can all be realized as submodules of this direct sum \cite[Proposition 2.1.2]{JoSu97}:

\begin{proposition}\label{prop:big_module}
Let $M$ be a finite von Neumann algebra which has separable predual, let $\tau$ be a faithful normal tracial state $\tau$ on $M$ and let $\mathcal{K}$ be a separable infinite-dimensional Hilbert space. Then every separable Hilbert $M$-module is isomorphic to a submodule of $L^2(M,\tau) \otimes \mathcal{K}$.
\end{proposition}

Now let $M'$ be the commutant of $M$ in $\mathcal{B}(L^2(M,\tau))$ which is also finite. The trace $\tau$ extends to a state on $\B(L^2(M,\tau))$ via $a \mapsto \langle a \Omega, \Omega \rangle$ for $a \in \B(L^2(M,\tau))$ and the restriction to $M'$ is also a faithful normal tracial state. Let $\Hi$ be a separable Hilbert $M$-module, which we embed into $L^2(M,\tau) \otimes \mathcal{K}$ as in \Cref{prop:big_module}. Denote by $p$ the projection of $L^2(M,\tau) \otimes \mathcal{K}$ onto $\Hi$. Then this projection is in the commutant of $M$ on the Hilbert space $L^2(M,\tau) \otimes \mathcal{K}$. This commutant equals $M' \otimes \mathcal{B}(\mathcal{K})$. The latter is a semifinite von Neumann algebra with center-valued trace $\Phi$ as defined in \eqref{eq:semi_finite}. We define the \emph{center-valued von Neumann dimension} of $\Hi$ to be
\begin{equation}
    \cdim_M \Hi = \Phi(p) = \sum_i \Tr(p_{ii}) \label{eq:cdim}
\end{equation}
where $\Tr$ denotes the canonical center-valued trace on $M'$. It follows from part (1) of \Cref{prop:projections_semifinite} that this definition is independent of the chosen projection. Moreover we define the scalar-valued von Neumann dimension of $\Hi$ to be
\begin{equation}
    \dim_M \Hi = (\tau \otimes \tr)(p) = \sum_i \tau(p_{ii}) \label{eq:dim} .
\end{equation}
Note that these two notions of dimension coincide precisely when $M$ is a factor.

\begin{proposition}\label{prop:cdim_iso}
Let $M$ be a finite von Neumann algebra with separable predual equipped with a faithful normal tracial state $\tau$, and let $\Hi$ and $\Hi'$ be separable Hilbert $M$-modules.
\begin{enumerate}
    \item If $\Hi$ is isomorphic to a submodule of $\Hi'$ then $\cdim_M \Hi \leq \cdim_M \Hi'$, and if $\Hi \cong \Hi'$ then $\cdim_M \Hi = \cdim_M \Hi'$.
    \item Suppose that $\dim_M \Hi, \dim_M \Hi' < \infty$. Then $\Hi$ is isomorphic to a submodule of $\Hi'$ if and only if $\cdim_M \Hi \leq \cdim_M \Hi'$, and $\Hi \cong \Hi'$ if and only if $\cdim_M \Hi = \cdim_M \Hi'$.
    \item $\cdim_M (\Hi \oplus \Hi') = \cdim_M \Hi + \cdim_M \Hi'$.
\end{enumerate}
\end{proposition}

\begin{proof}
Let $\cdim_M \Hi = \Phi(p)$ for a projection $p \in M' \otimes \mathcal{B}(\mathcal{K})$. Since $\tau \otimes \id$ is a faithful, semifinite, normal extended scalar-valued trace on $M \otimes \mathcal{B}(\mathcal{K})$, $(\tau \otimes \id)(p) < \infty$ implies that $p$ is a finite projection. Thus, since $\dim_M \Hi = (\tau \otimes \id)(p)$ both (1) and (2) follow directly from \Cref{prop:projections_semifinite}. (3) follows from the additivity of $\Phi$.
\end{proof}

\section{Twisted group von Neumann algebras}\label{sec:twisted_group}

\subsection{2-cocycles and projective representations}\label{subsec:cocycles}

In this section $G$ denotes a locally compact group with identity $e$. A \emph{2-cocycle} on $G$ is a Borel measurable function $\sigma \colon G \times G \to \T$ that satisfies the following properties:
\begin{align}
    \sigma(x, y) \sigma(xy, z) &= \sigma(x, yz) \sigma(y, z) \;\;\; \text{for $x,y,z \in G$,} \label{eq:cocycle1} \\
    \sigma(e,e) &= 1.
\end{align}
If $\sigma$ is a 2-cocycle then the pointwise conjugate $\overline{\sigma}$ is also a 2-cocycle.

An element $x \in G$ is called \emph{$\sigma$-regular} if $\sigma(x,y) = \sigma(y,x)$ whenever $y$ commutes with $x$. If $x$ is $\sigma$-regular then every element in the conjugacy class $C_{x}$ of $x$ is $\sigma$-regular, hence it makes sense to talk about $\sigma$-regular conjugacy classes. We say that $(G,\sigma)$ satisfies \emph{Kleppner's condition} if the only $\sigma$-regular finite conjugacy class is the trivial one.

A \emph{$\sigma$-projective unitary representation} of $G$ on a Hilbert space $\mathcal{H}$ is a map $\pi \colon G \to \mathcal{U}(\mathcal{H})$ that satisfies the following property:
\begin{equation*}
    \pi(x)\pi(y) = \sigma(x,y) \pi(xy) \;\;\; \text{for all $x,y \in G$.}
\end{equation*}
We also require that $\pi$ is \emph{strongly continuous}: That is, the map $G \to \mathcal{H}$ given by $x \mapsto \pi(x) \xi$ is continuous for every $\xi \in \mathcal{H}$.

Every 2-cocycle $\sigma$ on $G$ comes with two natural $\sigma$-projective representations: The \emph{$\sigma$-projective left regular representation of $G$} is the representation $\lambda_{\sigma}$ of $G$ on $L^2(G)$ given by
\begin{equation*}
    \lambda_{\sigma}(x)f(y) = \sigma(x,x^{-1}y)f(x^{-1}y) \;\;\; \text{for $x,y \in G$ and $f \in L^2(G)$,}
\end{equation*}
while the \emph{$\sigma$-projective right regular representation of $G$} is the representation $\rho_{\sigma}$ of $G$ on $L^2(G)$ given by
\begin{equation*}
    \rho_{\sigma}(x)f(y) = \sigma(y,x) f(yx) \;\;\; \text{for $x,y \in G$ and $f \in L^2(G)$.} 
\end{equation*}

Using the 2-cocycle identity \eqref{eq:cocycle1} one shows that the following conjugation identity holds for projective representations:
\begin{equation}
    \pi(y)^*\pi(x)\pi(y) = \sigma(x,y) \overline{ \sigma(y,y^{-1} x y) } \pi(y^{-1} x y) \;\;\; \text{for all $x,y \in G$.} \label{eq:projective_conjugation}
\end{equation}
The associated function $\tilde{\sigma} \colon G \times G \to \T$ given by
\begin{equation}
    \tilde{\sigma}(x,y) = \sigma(x,y) \overline{ \sigma(y,y^{-1} x y) } \;\;\; \text{for $x,y \in G$,} \label{eq:tilde_cocycle}
\end{equation}
has some important properties that form the basis for later 2-cocycle computations:

\begin{lemma}\label{lem:cocycle_conj}
Let $\sigma$ be a 2-cocycle on a locally compact group $G$. The following hold for all $x,y,z \in G$ where $\tilde{\sigma}$ is as defined in \eqref{eq:tilde_cocycle}:
\begin{align}
    \tilde{\sigma}(x,yz) &= \tilde{\sigma}(x,y) \tilde{\sigma}(y^{-1}xy, z) , \label{eq:tilde1} \\
    \tilde{\sigma}(x,y^{-1}) &= \overline{ \tilde{\sigma}(yxy^{-1}, y) } , \label{eq:tilde2} \\
    \tilde{\sigma}(x,y) &= \tilde{\sigma}(x,y') \;\;\; \text{if $x$ is $\sigma$-regular and $y^{-1}xy = y'^{-1}xy'$} \label{eq:tilde3}
\end{align}
\end{lemma}

\begin{proof}
Using the 2-cocycle identity \eqref{eq:cocycle1} repeatedly, we obtain
\begin{align*}
    &\tilde{\sigma}(x,y) \tilde{\sigma}(y^{-1}xy,z) \\
    &= \sigma(x, y) \overline{ \sigma( y, y^{-1}xy) } \sigma( y^{-1}xy, z) \overline{ \sigma( z, z^{-1} y^{-1} x y z )} \\
    &= \sigma(x,y) \sigma( xy,z) \sigma( y^{-1} x y, z) \overline{ \sigma( y, y^{-1} x y) }  \overline{ \sigma( xy,z) } \overline{ \sigma( z, z^{-1} y^{-1} x y z)} \\
    &= \sigma(x, yz) \sigma( y,z) \sigma( y^{-1} x y, z) \overline{ \sigma( y, y^{-1} x y z) } \overline{ \sigma( y^{-1} x y, z )} \overline{ \sigma(z, z^{-1} y^{-1} x y z) } \\
    &= \sigma(x, yz) \sigma( y,z)  \overline{ \sigma( y, y^{-1} x y z) } \overline{ \sigma(z, z^{-1} y^{-1} x y z) } \\
    &= \sigma(x, yz) \sigma( y,z) \overline{ \sigma( y, z) } \overline{ \sigma( yz, z^{-1} y^{-1} x y z  ) } \\
    &= \sigma(x, yz) \overline{ \sigma( yz, z^{-1} y^{-1} x y z  ) } \\
    &= \tilde{\sigma}(x,yz) .
\end{align*}
This proves \eqref{eq:tilde1}. Now \eqref{eq:tilde2} follows from the following special case of \eqref{eq:tilde1}:
\[ 1 = \tilde{\sigma}(x,y^{-1} y ) = \tilde{\sigma}(x,y^{-1}) \tilde{\sigma}(yxy^{-1},y) .\]
To prove \eqref{eq:tilde3}, note that $y^{-1} x y = y'^{-1} x y'$ implies that $y y'^{-1} x y' y^{-1} = x$, hence $x$ commutes with $yy'^{-1}$. Since $x$ is assumed to be $\sigma$-regular, we obtain
\[ \tilde{\sigma}(x,y'y^{-1}) = \sigma(x,y'y^{-1}) \overline{ \sigma(y'y^{-1}, y y'^{-1} x y' y^{-1})} = \sigma(x,y'y^{-1}) \overline{ \sigma(y'y^{-1}, x)} = 1 .\]
Applying \eqref{eq:tilde1} and \eqref{eq:tilde2} we get
\[ \tilde{\sigma}(x,y') \overline{ \tilde{\sigma}(x,y)}  =  \tilde{\sigma}(x,y') \overline{ \tilde{\sigma}( yy'^{-1} x y' y^{-1}, y) } =  \tilde{\sigma}(x,y') \tilde{\sigma}(y'^{-1} x y', y^{-1}) = \tilde{\sigma}(x, y'y^{-1}) = 1. \]
Hence $\tilde{\sigma}(x,y) = \tilde{\sigma}(x,y')$.
\end{proof}

Note that when $G$ is abelian, the identity \eqref{eq:tilde1} of \Cref{lem:cocycle_conj} reduces to the fact that $(x,y) \mapsto \sigma(x,y) \overline{ \sigma(y,x) }$ is a bicharacter on $G$, which is well-known (see e.g.\ \cite[Lemma 7.1]{Kl65}).

\subsection{The center-valued trace on twisted group von Neumann algebras}\label{subsec:twisted_group}

Let $\Gamma$ be a discrete group with a 2-cocycle $\sigma$ and denote by $\lambda_{\sigma}$ (resp.\ $\rho_{\sigma}$) the $\sigma$-twisted left (resp.\ right) regular representation of $\Gamma$. We then define the following two associated von Neumann algebras on $\ell^2(\Gamma)$:
\begin{align*}
    \vN(\Gamma,\sigma) = \lambda_{\sigma}(\Gamma)'', && \rvN(G,\sigma) = \rho_{\sigma}(\Gamma)''.
\end{align*}
The von Neumann algebra $\vN(\Gamma,\sigma)$ is called the \emph{$\sigma$-twisted group von Neumann algebra of $\Gamma$}. It is well-known that $\rvN(\Gamma,\overline{\sigma})$ is the commutant of $\vN(\Gamma,\sigma)$ on $\ell^2(\Gamma)$ (see \cite[Theorem 1]{Kl62}).

For the remainder of the section we set $M = \vN(\Gamma,\sigma)$ and $N = \rvN(\Gamma,\overline{\sigma})$. Let $\{ \delta_{\gamma} : \gamma \in \Gamma \}$ denote the usual orthonormal basis for $\ell^2(\Gamma)$. The map $\tau \colon \B(\ell^2(\Gamma)) \to \C$ given by
\[ \tau(a) = \langle a \delta_e, \delta_e \rangle \;\;\; \text{for $a \in M$} \]
restricts to a faithful normal tracial state on both $M$ and $N$, which shows that these von Neumann algebras are finite. Moreover, the GNS construction $L^2(M,\tau)$ with respect to $\tau$ is canonically isomorphic to $\ell^2(\Gamma)$ with cyclic, separating vector $\Omega = \delta_e$.

\begin{proposition}\label{thm:cvt}
The center-valued trace on $\vN(\Gamma,\sigma)$ is given as follows: If $\gamma \in \Gamma$ is such that the conjugacy class $C_{\gamma}$ is $\sigma$-regular and finite, say $C_{\gamma} = \{ \beta_1^{-1} \gamma \beta_1 , \ldots, \beta_k^{-1} \gamma \beta_k \}$, then
\[ \Tr(\lambda_{\sigma}(\gamma)) = |C_{\gamma}|^{-1} \sum_{j=1}^k \sigma(\gamma,\beta_j) \overline{ \sigma(\beta_j, \beta_j^{-1} \gamma \beta_j) } \lambda_{\sigma}(\beta_j^{-1} \gamma \beta_j). \]
Otherwise $\Tr(\lambda_{\sigma}(\gamma)) = 0$.
\end{proposition}

\begin{proof}
To ease notation set $\lambda = \lambda_{\sigma}$. Let $\gamma \in \Gamma$. For any $\beta \in \Gamma$ we have by \eqref{eq:projective_conjugation} that
\begin{align}
    \lambda(\beta)^*\lambda(\gamma)\lambda(\beta) = \tilde{\sigma}(\gamma,\beta) \lambda(\beta^{-1} \gamma \beta ) . \label{eq:proj_conj_proof}
\end{align}
Taking center-valued traces, we obtain
\begin{equation}
    \Tr(\lambda(\gamma)) = \tilde{\sigma}(\gamma,\beta) \Tr(\lambda(\beta^{-1} \gamma \beta )) . \label{eq:cvt_eq}
\end{equation}
If $C_{\gamma}$ is infinite then $\lambda(\gamma) = 0$ by \cite[Lemma 2.2]{Om14}. If $C_{\gamma}$ is not $\sigma$-regular then there exists $\beta \in \Gamma$ that commutes with $\lambda$ yet $\sigma(\gamma,\beta) \neq \sigma(\beta, \gamma)$. But then \eqref{eq:cvt_eq} gives $\Tr(\lambda(\gamma)) =  \sigma(\gamma,\beta) \overline{\sigma(\beta,\gamma)} \Tr(\lambda(\gamma))$ which implies that $\Tr(\lambda(\gamma)) = 0$.

Suppose now that $C_{\gamma}$ is both $\sigma$-regular and finite, say $C_{\gamma} = \{ \beta_1^{-1} \gamma \beta_1, \ldots, \beta_k^{-1} \gamma \beta_k \}$. Summing \eqref{eq:proj_conj_proof} over all $\beta \in C_{\gamma}$ we get
\begin{equation}
    \sum_{j=1}^k \lambda(\beta_j)^*\lambda(\gamma)\lambda(\beta_j) = \sum_{j=1}^k \tilde{\sigma}(\gamma,\beta_j) \lambda(\beta_j^{-1} \gamma \beta_j ) . \label{eq:backto}
\end{equation}
Denote the right side of the above equality by $S$. We claim that $S$ is in the center of $\vN(\Gamma,\sigma)$. Indeed, if $\beta \in \Gamma$ we use \Cref{lem:cocycle_conj} \eqref{eq:tilde1} to compute that
\begin{align*}
    \lambda(\beta)^* S \lambda(\beta) &= \sum_{j=1}^k \tilde{\sigma}(\gamma,\beta_j) \lambda(\beta)^* \lambda( \beta_j^{-1} \gamma \beta_j) \lambda(\beta) \\
    &= \sum_{j=1}^k \tilde{\sigma}(\gamma,\beta_j) \tilde{\sigma}(\beta_j^{-1} \gamma \beta_j, \beta) \lambda( \beta^{-1} \beta_j^{-1} \gamma \beta_j \beta ) \\
    &= \sum_{j=1}^k \tilde{\sigma}(\gamma,\beta_j \beta) \lambda( \beta^{-1} \beta_j^{-1} \gamma \beta_j \beta ) .
\end{align*}
The set $\{ \beta^{-1} \beta_j^{-1} \gamma \beta_j \beta : 1 \leq j \leq k \}$ is equal to $C_{\gamma}$, say
\[ \beta^{-1} \beta_j^{-1} \gamma \beta_j \beta = \beta_{m_j}^{-1} \gamma \beta_{m_j} \]
for each $1 \leq j \leq k$, where $\{ m_1, \ldots, m_k \} = \{ 1, \ldots, k \}$. Since $\gamma$ is $\sigma$-regular, \Cref{lem:cocycle_conj} (3) implies that $\tilde{\sigma}(\gamma, \beta_j \beta) = \tilde{\sigma}(\gamma, \beta_{m_j})$ for each $j$, hence
\[ \lambda(\beta)^* S \lambda(\beta) = \sum_{j=1}^k \tilde{\sigma}(\gamma, \beta_{m_j}) \lambda( \beta_{m_j}^{-1} \gamma \beta_{m_j}) = S .\]
We conclude that $S$ is in the center of $\vN(\Gamma,\sigma)$. Taking the center-valued trace on both sides of \eqref{eq:backto} we now obtain
\begin{align*}
    |C_{\gamma}|\Tr(\lambda(\gamma)) = \sum_{j=1}^k \Tr\big( \lambda(\beta_j)^*\lambda(\gamma)\lambda(\beta_j) \big) = \sum_{j=1}^k \tilde{\sigma}(\gamma,\beta_j) \lambda(\beta_j^{-1} \gamma \beta_j).
\end{align*}
Dividing by $|C_{\gamma}|$ on both sides finishes the proof.
\end{proof}

As a corollary, we obtain the following well-known characterization of when $\vN(\Gamma,\sigma)$ is a factor.

\begin{corollary}
The von Neumann algebra $\vN(\Gamma,\sigma)$ is a factor if and only if $(\Gamma,\sigma)$ satisfies Kleppner's condition.
\end{corollary}

\begin{proof}
$\vN(\Gamma,\sigma)$ is a factor if and only if the center-valued trace reduces to the canonical faithful tracial state on $\vN(\Gamma,\sigma)$. By the expression in \Cref{thm:cvt} this happens if and only if the only $\sigma$-regular finite conjugacy class is $\{ e \}$.
\end{proof}

An analogous computation to that of the proof of \Cref{thm:cvt} shows that the center-valued trace on $\rvN(\Gamma,\sigma)$ is given by
\begin{equation}
    \Tr(\rho_{\sigma}(\gamma)) = |C_{\gamma}|^{-1} \sum_{j=1}^k \sigma(\gamma , \beta_j) \overline{ \sigma( \beta_j, \beta_j^{-1} \gamma \beta_j ) }  \rho_{\sigma}(\beta_j^{-1} \gamma \beta_j) \label{eq:cvt_r}
\end{equation}
if $C_{\gamma}$ is finite and $\sigma$-regular and $\Tr(\rho_{\sigma}(\gamma)) = 0$ otherwise. Note that for $\rvN(\Gamma,\overline{\sigma}) = \vN(\Gamma,\sigma)'$ one needs to conjugate $\sigma$ in the above equation.

\subsection{Fourier coefficients and positivity}

We can express the vectors $\delta_{\gamma}$ for $\gamma \in \Gamma$ in terms of $\lambda_{\sigma}$ and $\rho_{\overline{\sigma}}$ as
\begin{equation}
    \delta_\gamma = \lambda_{\sigma}(\gamma) \delta_e = \rho_{\overline{\sigma}}(\gamma)^* \delta_e . \label{eq:delta_gamma}
\end{equation}

The $\sigma$-twisted convolution of two functions $f,g \colon \Gamma \to \C$ is given by
\[ (f *_{\sigma} g)(\gamma) = \sum_{\gamma' \in \Gamma} \sigma(\gamma',\gamma'^{-1}\gamma)  f(\gamma')g(\gamma'^{-1}\gamma) = \sum_{\gamma' \in \Gamma} f(\gamma') \lambda_{\sigma}(\gamma')g(\gamma) \;\;\; \text{for $\gamma \in \Gamma$.} \]
Given $a \in \vN(\Gamma,\sigma)$, then the \emph{Fourier coefficient} of $a$ is the element $\widehat{a} = a \delta_e$ of $\ell^2(\Gamma)$. Since $\delta_e$ is a separating vector for $\vN(\Gamma,\sigma)$ on $\ell^2(\Gamma)$, it follows that $a$ is uniquely determined by $\widehat{a}$. Using twisted convolution, we can describe how $a$ acts on $\ell^2(\Gamma)$ via its Fourier coefficient; see \cite[p.\ 343]{BeCo09}:
\begin{equation*}
    a f = \widehat{a} *_{\sigma} f \label{eq:fourier_conv} \;\;\; \text{for all $f \in \ell^2(\Gamma)$.}
\end{equation*}
In particular $\widehat{a} *_{\sigma} f \in \ell^2(\Gamma)$ for all $f \in \ell^2(\Gamma)$. Using \eqref{eq:delta_gamma}, the values of $\widehat{a}$ can be expressed as
\begin{equation*}
    \widehat{a}(\gamma) = \langle \widehat{a}, \delta_{\gamma} \rangle = \tau( \lambda_{\sigma}(\gamma)^* a ) = \tau( \rho_{\overline{\sigma}}(\gamma) a ) . \label{eq:fourier_trace}
\end{equation*}

A function $\phi \in \ell^\infty(\Gamma)$ is called \emph{$\sigma$-positive definite} if for all $\gamma_1, \ldots, \gamma_n \in \Gamma$ and $c_1, \ldots c_n \in \C$ we have
\[ \sum_{i,j} \sigma(\gamma_j \gamma_i^{-1}, \gamma_i) c_i \overline{ c_j} \phi(\gamma_j \gamma_i^{-1}) \geq 0 . \]

\begin{proposition}\label{prop:pos_def}
A function $\phi \in \ell^\infty(\Gamma)$ is $\sigma$-positive definite if and only if
\[ \langle \phi *_{\sigma} f, f \rangle \geq 0 \]
for all functions $f \colon \Gamma \to \C$ with finite support. In particular, $a \in \vN(\Gamma,\sigma)$ is positive if and only if $\widehat{a}$ is a $\sigma$-positive definite function.
\end{proposition}

\begin{proof}
For a function $f$ on $\Gamma$ of finite support we have that
\begin{align*}
    \langle \phi *_{\sigma} f, f \rangle &= \sum_{\gamma \in \Gamma} (\phi *_{\sigma} f)(\gamma) \overline{f(\gamma)} \\
    &= \sum_{\gamma,\gamma' \in \Gamma} \phi(\gamma') f(\gamma'^{-1} \gamma) \sigma(\gamma', \gamma'^{-1} \gamma) \overline{f(\gamma)} \\
    &= \sum_{\gamma,\gamma' \in \Gamma} \sigma(\gamma \gamma'^{-1}, \gamma') f(\gamma') \overline{f(\gamma)} \phi(\gamma \gamma'^{-1}) .
\end{align*}
Letting the support of $f$ be $\{ \gamma_1, \ldots, \gamma_n \}$ and setting $c_i = f(\gamma_i)$ for each $1 \leq i \leq n$, the above expression becomes exactly the expression in the definition of $\sigma$-positivity.

The positivity of $a \in \vN(\Gamma,\sigma)$ is equivalent to $\langle af, f \rangle = \langle \widehat{a} *_{\sigma} f, f \rangle \geq 0$ for all $f \in \ell^2(\Gamma)$. By writing $f$ as a limit of functions on $\Gamma$ of finite support the last condition is equivalent to $\widehat{a}$ being $\sigma$-positive definite as shown above.
\end{proof}

\section{Hilbert modules from square-integrable representations}\label{sec:dimension}

Throughout this section $G$ denotes a second countable, unimodular, locally compact group, $\sigma$ denotes a 2-cocycle on $G$, and $\Gamma$ denotes a lattice in $G$. We fix a Haar measure on $G$ and simply write $\dif{x}$ when we integrate with respect to it. Soon $(\pi, \Hip)$ will denote a $\sigma$-projective unitary representation of $G$ which is irreducible and square-integrable as defined in \Cref{prop:si_unimodular}. We will denote by $\lambda_{\sigma}^G$ (resp.\ $\lambda_{\sigma}^{\Gamma}$) the $\sigma$-projective left regular representation of $G$ (resp.\ of $\Gamma$). We will also let $M = \vN(\Gamma,\sigma)$ and $N = \rvN(\Gamma,\overline{\sigma})$.

\subsection{Lattices in locally compact groups}\label{subsec:subgroups_lcgroups}

Suppose $\Gamma$ is a lattice in $G$. Then there exists a Borel measurable set $B \subseteq G$ such that the collection $\{ \gamma B : \gamma \in \Gamma \}$ forms a partition of $G$ \cite{Ma52}. Such a set $B$ is called a \emph{fundamental domain} for $\Gamma$ in $G$. Weil's formula relates integration over $G$ to integration over $\Gamma$ and $B$:
\begin{align*}
    \int_G f(x) \dif{x} = \sum_{\gamma \in \Gamma} \int_B f(\gamma y) \dif{y} \;\;\; \text{for all $f \in L^1(G)$.} 
\end{align*}
In general, there are many fundamental domains for $\Gamma$ in $G$, but they all have the same measure. This number is called the \emph{covolume} of $\Gamma$ in $G$ and is denoted by $\vol(G/\Gamma)$. If $B$ has finite measure then $\Gamma$ is called a \emph{lattice} in $G$.

Now set $\mathcal{K} = L^2(B)$ and fix a 2-cocycle $\sigma$ on $G$. Let $M = \vN(\Gamma,\sigma)$ be the $\sigma$-twisted group von Neumann algebra of $\Gamma$ with its canonical trace $\tau$. Then according to \Cref{subsec:twisted_group}, the GNS construction of $(M,\tau)$ can be naturally identified with $\ell^2(\Gamma)$. Because of our assumptions on $G$, the Hilbert space $\mathcal{K}$ is separable and infinite-dimensional if $G$ is infinite. Consequently the Hilbert $M$-module $\ell^2(\Gamma) \otimes \mathcal{K}$ is exactly the module into which every separable Hilbert $M$-module embeds according to \Cref{prop:big_module}. If $G$ is finite then $\vN(\Gamma,\sigma)$ is finite-dimensional and in this case every Hilbert $M$-module also embeds into $\ell^2(\Gamma) \otimes \mathcal{K}$. The following proposition shows that $\ell^2(\Gamma) \otimes \mathcal{K}$ can be naturally identified with $L^2(G)$.

\begin{proposition}\label{prop:big_module_l2}
The von Neumann algebra $\lambda_{\sigma}^G(\Gamma)''$ on $L^2(G)$ is isomorphic to $\vN(\Gamma,\sigma)$. Moreover, the Hilbert $\vN(\Gamma,\sigma)$-module $L^2(G)$ is isomorphic to $\ell^2(\Gamma) \otimes \mathcal{K}$ via the map $U \colon \ell^2(\Gamma) \otimes \mathcal{K} \to L^2(G)$ given by
\begin{equation}
    U(f \otimes g)(\gamma y) = \sigma(\gamma,y) f(\gamma) g(y) \;\;\; \text{for $\gamma \in \Gamma$ and $y \in B$.}
\end{equation}
\end{proposition}

\begin{proof}
That $U$ is a well-defined unitary operator follows from the fact that $B$ is a fundamental domain for $\Gamma$ in $G$.

Let $f \in \ell^2(\Gamma)$ and $g \in \mathcal{K}$. The following computation shows that $U$ intertwines $\lambda_{\sigma}^{\Gamma} \otimes I$ and $\lambda_{\sigma}^G$, where $\lambda_{\sigma}^{\Gamma}$ denotes the $\sigma$-twisted left regular representation of $\Gamma$:
\begin{align*}
    \lambda_{\sigma}^G(\gamma) U(f \otimes g)(\gamma' y) &= \sigma(\gamma,\gamma^{-1} \gamma' y) U(f \otimes g)(\gamma^{-1} \gamma' y) \\
    &= \sigma(\gamma,\gamma^{-1} \gamma' y) \sigma(\gamma^{-1} \gamma' ,y) f(\gamma^{-1} \gamma')g(y) \\
    &= \sigma(\gamma, \gamma^{-1} \gamma') \sigma( \gamma \gamma^{-1} \gamma', y)  f(\gamma^{-1} \gamma')g(y) \\
    &= \sigma(\gamma', y) (\lambda_{\sigma}^{\Gamma}(\gamma)f \otimes g )(\gamma' y) \\
    &= U( \lambda_{\sigma}^{\Gamma}(\gamma) f \otimes g)(\gamma' y) .
\end{align*}
Thus the map $\vN(\Gamma,\sigma) \to \lambda_{\sigma}^G(\Gamma)''$ given by $a \mapsto U (a \otimes I) U^*$ is an isomorphism of von Neumann algebras and the Hilbert $\vN(\Gamma,\sigma)$-modules $\ell^2(\Gamma) \otimes \mathcal{K}$ and $L^2(G)$ are isomorphic.
\end{proof}

Let $(e_i)_{i \in \N}$ be an orthonormal basis for $\mathcal{K}$, and denote by $\tilde{e_i}$ the extension of $e_i$ to the whole of $G$ by zero outside of $B$. Then $U(\delta_e \otimes e_i) = \tilde{e_i}$. Since $( \delta_{\gamma} )_{\gamma \in \Gamma}$ is an orthonormal basis for $\ell^2(\Gamma)$, it follows that $( \delta_{\gamma} \otimes e_i )_{\gamma \in \Gamma, i \in \N}$ is an orthonormal basis for $\ell^2(\Gamma) \otimes \mathcal{K}$. Since
\[ U( \delta_{\gamma} \otimes e_i) = U( \lambda_{\sigma}^{\Gamma}(\gamma) \delta_e \otimes e_i ) = \lambda_{\sigma}^G(\gamma) (\delta_e \otimes e_i) = \lambda_{\sigma}^G(\gamma) \tilde{e_i} \]
it follows from \Cref{prop:big_module_l2} that $( \lambda_{\sigma}^G(\gamma) \tilde{e_i} )_{\gamma \in \Gamma, i \in \N }$ is an orthonormal basis for $L^2(G)$.

\subsection{Orthogonality relations}

The following orthogonality relation for the matrix coefficients of irreducible square-integrable representations forms the basis for our considerations, see \cite[Definition/Proposition 3.1]{Ra98} for a proof.

\begin{proposition}\label{prop:si_unimodular}
Let $\pi$ be a $\sigma$-projective irreducible unitary representation of a unimodular locally compact group $G$. The following are equivalent:
\begin{enumerate}
    \item There exist nonzero vectors $\xi, \eta \in \Hip$ such that $\int_G | \langle \xi, \pi(x) \eta \rangle |^2 \dif{x} < \infty$.
    \item For every $\xi , \eta \in \Hip$ we have that $\int_G | \langle \xi, \pi(x) \eta \rangle |^2 \dif{x} < \infty$.
    \item $\pi$ is a subrepresentation of the $\sigma$-twisted left regular representation of $G$.
\end{enumerate}
In case any (and hence all) of the above assumptions hold, then there exists a number $d_{\pi} > 0$ called the \emph{formal dimension} of $\pi$ such that
\begin{equation}
    \int_G \langle \xi, \pi(x) \eta \rangle \overline{ \langle \xi', \pi(x) \eta' \rangle } \dif{x} = d_{\pi}^{-1} \langle \xi, \xi' \rangle \overline{ \langle \eta,\eta' \rangle } \label{eq:schur}
\end{equation}
for all $\xi, \eta \in \Hip$.
\end{proposition}

Representations satisfying any of the equivalent conditions in \Cref{prop:si_unimodular} are called \emph{square-integrable}. We now fix a square-integrable, $\sigma$-projective, irreducible, unitary representation $\pi$ of $G$.

We detail the passage from $(2)$ to $(3)$ in \Cref{prop:si_unimodular} as it will be relevant in this section. For $\xi,\eta \in \Hip$ we define the \emph{(generalized) wavelet transform} $V_{\eta} \xi \colon G \to \C$ by
\begin{equation*}
    V_{\eta} \xi(x) = \langle \xi, \pi(x) \eta \rangle \;\;\; \text{for all $x \in G$.}
\end{equation*}
From the assumption that $(2)$ in \Cref{prop:si_unimodular} holds it follows that $V_{\eta}$ maps $\Hip$ into $L^2(G)$. Moreover, $V_{\eta} \xi$ intertwines $\pi$ and the $\sigma$-twisted left regular representation as can be seen from the following calculation for $x,y \in G$:
\begin{align*}
    V_{\eta} (\pi(x) \xi)(y) &= \langle \pi(x) \xi, \pi(y) \eta \rangle \\
    &= \langle \xi, \pi(x)^* \pi(y) \eta \rangle \\
    &= \langle \xi, \overline{\sigma(x,x^{-1})} \sigma(x^{-1},y) \pi(x^{-1}y) \rangle \\
    &= \sigma(x,x^{-1}y) V_{\eta} \xi(x^{-1}y) \\
    &= \lambda_{\sigma}(x) V_{\eta} \xi(y) .
\end{align*}
If one sets $\eta = \eta'$ to be a unit vector in \eqref{eq:schur} one obtains
\[ \langle V_{\eta} \xi, V_{\eta} \xi' \rangle = d_{\pi}^{-1} \langle \xi, \xi' \rangle \;\;\; \text{for all $\xi,\xi' \in \Hip$.} \]
It follows that the map $d_{\pi}^{1/2} V_{\eta}$ is an isometry from $\Hip$ to $L^2(G)$. Since $V_{\eta}$ is also an intertwiner, $\pi$ is a subrepresentation of $\lambda_{\sigma}$. This establishes $(3)$ of \Cref{prop:si_unimodular}.

Now using the fact that $d_{\pi}^{1/2} V_{\eta}$ is an isometric intertwiner between $\pi$ and $\lambda_{\sigma}$, $V_{\eta} \Hip$ is a submodule of the Hilbert $\vN(\Gamma,\sigma)$-module $L^2(G)$ from \Cref{prop:big_module_l2}. Thus $\Hip$ becomes a Hilbert $\vN(\Gamma,\sigma)$-module isomorphic to $V_{\eta} \Hip$ via the action
\[ a \xi \coloneqq V_{\eta}^* a V_{\eta} \xi \;\;\; \text{for $a \in \vN(\Gamma,\sigma)$ and $\xi \in \Hip$.} \]

\subsection{Computing the center-valued von Neumann dimension}\label{subsec:computation}

In this subsection we will compute the center-valued von Neumann dimension of $\Hip$ as a Hilbert $\vN(\Gamma,\sigma)$-module. Our approach will be a modification of the approach in \cite{Be04} with necessary changes needed to incorporate the 2-cocycle $\sigma$.

As before let $M = \vN(\Gamma,\sigma)$ and $N = \rvN(\Gamma,\overline{\sigma})$.  Let $\eta$ be any unit vector in $\Hip$. By the discussion in the previous section, $\Hip$ has the structure of a Hilbert $\vN(\Gamma,\sigma)$-module defined using the wavelet transform $V_{\eta}$. Denoting by $U$ the intertwiner of \Cref{prop:big_module_l2}, we have that $U^* V_{\eta} \Hip \subseteq \ell^2(\Gamma) \otimes \mathcal{K}$, so let $p$ denote the projection of $\ell^2(\Gamma) \otimes \mathcal{K}$ onto $U^* V_{\eta} \Hip$. Then $\cdim_M \Hip = \Phi(p)$ where $\Phi$ is the faithful, semi-finite, normal extended center-valued trace of $N \otimes \mathcal{B}(\mathcal{K})$ as defined as in \eqref{eq:semi_finite}. Set $\tilde{p} = U p U^*$, i.e.\ $\tilde{p}$ is the orthogonal projection of $L^2(G)$ onto $V_{\eta} \Hip$.

Now $\Phi(p) = \sum_i \Tr(p_{ii})$ where each $\Tr(p_{ii})$ is a positive operator in $\mathcal{Z}(N) = \mathcal{Z}(M)$. By \Cref{prop:pos_def} they are all given by $\Tr(p_{ii}) f = \phi_i *_{\sigma} f$ for $f \in \ell^2(\Gamma)$ where $\phi_i$ is the Fourier coefficient of $\Tr(p_{ii})$. The values of $\phi_i$ can be expressed as $\phi_i(\gamma) = \tau( \rho_{\overline{\sigma}}(\gamma) \Tr(p_{ii}))$ for $\gamma \in \Gamma$. Summing these values over $i$ we obtain
\[ \phi(\gamma) \coloneqq \sum_i \tau(\rho_{\overline{\sigma}}(\gamma) \Tr(p_{ii})) = \tau \Big( \rho_{\overline{\sigma}}(\gamma) \sum_i \Tr(p_{ii}) \Big) = \tau( \rho_{\overline{\sigma}}(\gamma) \Phi(p)) .\]
Note that for $a \in M$, $a \geq 0$, we have that
\[ | \tau(a \, \Phi(p)) | \leq \sum_i | \tau( a \Tr(p_{ii})) | \leq \sum_i \|  a \| \tau( \Tr(p_{ii})) = \| a \| \dim_M \Hip .\]
Hence, if $\dim_M \Hip < \infty$, then the function $\phi$ is well-defined and in $\ell^\infty(\Gamma)$. Moreover, as a (possibly unbounded) operator $\Phi(p)$ acts as
\[ \Phi(p) f = \sum_i \Tr(p_{ii}) f = \sum_i \phi_i *_{\sigma} f = \phi *_{\sigma} f \]
for $f \in \ell^2(\Gamma)$. Thus we can describe $\cdim_M \Hip$ by describing $\phi$, and that is the content of the following theorem:

\begin{theorem}\label{thm:cvfd}
The Hilbert $\vN(\Gamma,\sigma)$-module $\Hip$ has scalar-valued von Neumann dimension equal to
\[ \dim_{\vN(\Gamma,\sigma)} \Hip = d_{\pi} \vol(G/\Gamma) \]
which is finite since $\Gamma$ is a lattice in $G$. Furthermore, the center-valued von Neumann dimension of $\Hip$ is the (possibly unbounded) operator on $\ell^2(\Gamma)$ given by $f \mapsto f *_{\sigma} \phi$ where $\phi \in \ell^\infty(\Gamma)$ is the function
\[ \phi(\gamma) = \begin{cases} \displaystyle  \frac{d_{\pi}}{|C_{\gamma}|} \int_{G/\Gamma_{\gamma}} \overline{\sigma(\gamma,y)}\sigma(y,y^{-1} \gamma y) V_{\eta} \eta(y^{-1} \gamma y) \dif{(y \Gamma_{\gamma})}   & \text{if $C_{\gamma}$ is $\sigma$-regular and finite,} \\ 0 & \text{otherwise.} \end{cases} \]
Here $C_{\gamma}$ is the conjugacy class of $\gamma$ in $\Gamma$, $\Gamma_{\gamma}$ is the centralizer of $\gamma$ in $\Gamma$, and $\eta$ is any unit vector in $\Hip$.
\end{theorem}

To prepare for the proof of \Cref{thm:cvfd} we need two lemmas.

\begin{lemma}\label{lem:integral_B}
The scalar-valued von Neumann dimension of $_M \Hip$ is given by $\dim_M \Hip = d_{\pi} \vol(G/\Gamma)$. Moreover, if $\Gamma$ is a lattice in $G$ and $\gamma \in \Gamma$ then
\begin{equation}
     \sum_i \tau( \rho_{\overline{\sigma}}^{\Gamma}(\gamma) p_{ii}) = d_{\pi} \int_B \overline{\sigma(\gamma,y)}\sigma(y,y^{-1} \gamma y) V_{\eta} \eta(y^{-1} \gamma y) \dif{y} , \label{eq:lemma_equation}
\end{equation}
where $B$ is a fundamental domain for $\Gamma$ in $G$.
\end{lemma}

\begin{proof}
We use the notation of \Cref{prop:big_module_l2} and the following dicsussion. Thus we let $(e_i)_i$ be an orthonormal basis for $\mathcal{K} = L^2(B)$, and we denote by $\tilde{e_i}$ the extension by zero of $e_i$ to all of $G$. Furthermore, let $(\eta_j)_j$ be an orthonormal basis for $\Hip$. Let $\eta \in \Hip$ have unit norm and set $g_j = d_{\pi}^{1/2} V_{\eta} \eta_j$. Then $(g_j)_j$ is an orthonormal basis for $V_{\eta} \Hip$ since $d_{\pi}^{1/2} V_{\eta}$ is an isometry. As before $\tilde{p}$ denotes the projection of $L^2(G)$ onto $V_{\eta} \Hip$. We also let $q$ denote the projection of $L^2(G)$ onto $U(\delta_e \otimes \mathcal{K})$, which has orthonormal basis $(\tilde{e_i})_{i=1}^\infty$.

We claim that the series
\begin{equation}
    \sum_{i,j} \langle \tilde{e_i}, g_j \rangle \langle g_j, \lambda_{\sigma}^G(\gamma) \tilde{e_i} \rangle \label{eq:abs_conv}
\end{equation}
is absolutely convergent when $\Gamma$ is a lattice in $G$. Indeed, using Cauchy--Schwarz we obtain
\begin{align*}
    \sum_{i,j} |\langle \tilde{e_i}, g_j \rangle \langle g_j, \lambda_{\sigma}^G(\gamma) \tilde{e_i} \rangle| &\leq \sum_i \left( \sum_j | \langle \tilde{e_i}, g_j \rangle |^2 \right)^{1/2} \left( \sum_j | \langle g_j, \lambda_{\sigma}^G(\gamma) \tilde{e_i} \rangle |^2 \right)^{1/2}  \\
    &= \sum_i \| \tilde{p} \tilde{e_i} \| \| \tilde{p} \lambda_{\sigma}^G(\gamma) \tilde{e_i} \| = \sum_i \| \tilde{p} \tilde{e_i} \|^2 = \sum_j \sum_i | \langle \tilde{e_i}, g_j \rangle |^2 \\
    &= \sum_j \| q( g_j) \|^2 = \sum_j \int_B | g_j(y) |^2 \dif{y} = d_{\pi} \sum_j \int_B | \langle \eta_j, \pi(y) \eta \rangle |^2 \dif{y} \\
    &= d_{\pi} \int_B \sum_j | \langle \eta_j, \pi(y) \eta \rangle |^2 \dif{y} = d_{\pi} \int_B \| \pi(y) \eta \|^2 \dif{y} \\
    &= d_{\pi} \vol(G/\Gamma) < \infty.
\end{align*}
Thus, we can compute the sum in which order we like. Summing over $j$ first we obtain
\begin{align*}
    \sum_{i,j}  \langle \tilde{e_i}, g_j \rangle \langle g_j, \lambda_{\sigma}^G(\gamma) \tilde{e_i} \rangle &= \sum_i \left\langle \sum_j \langle \tilde{e_i}, g_j \rangle g_j , \lambda_{\sigma}^G(\gamma) \tilde{e_i} \right\rangle = \sum_i \langle \tilde{p} \tilde{e_i}, \lambda_{\sigma}^G(\gamma) \tilde{e_i} \rangle \\
    &= \sum_i \langle p (\delta_e \otimes e_i), \delta_{\gamma} \otimes e_i \rangle = \sum_i \langle p_{ii} \delta_e, \delta_{\gamma} \rangle = \sum_i \tau( \rho_{\overline{\sigma}}^{\Gamma}(\gamma) p_{ii}) .
\end{align*}
If we sum over $i$ first instead, we obtain
\begin{align*}
    \sum_{i,j}  \langle \tilde{e_i}, g_j \rangle \langle g_j, \lambda_{\sigma}^G(\gamma) \tilde{e_i} \rangle &= \sum_j \left\langle \lambda_{\sigma}^G(\gamma)^* g_j, \sum_i \langle g_j, \tilde{e_i} \rangle \tilde{e_i} \right\rangle = \sum_j \langle \lambda_{\sigma}^G(\gamma)^* g_j, q(g_j) \rangle \\
    &= \sum_j \int_B \lambda_{\sigma}^G(\gamma)^* g_j(y) \overline{g_j(y)} \dif{y} = d_{\pi} \sum_j \int_B \langle \pi(\gamma)^* \eta_j, \pi(y) \eta \rangle \overline{ \langle \eta_j , \pi(y) \eta \rangle } \dif{y} \\
    &= d_{\pi}  \int_B \left\langle \sum_j \langle \pi(y) \eta, \eta_j \rangle \eta_j, \pi(\gamma)\pi(y) \eta \right\rangle \dif{y} \\
    &= d_{\pi} \int_B \langle \pi(y) \eta, \sigma(\gamma,y) \pi(\gamma y) \eta \rangle \dif{y} \\
    &= d_{\pi} \int_B \overline{ \sigma(\gamma,y)} \sigma( y, y^{-1} \gamma y ) V_{\eta} \eta(y^{-1} \gamma y) \dif{y} .
\end{align*}
This proves \eqref{eq:lemma_equation}. In particular, when $\gamma = e$ we get
\[ \dim_M \Hip = \sum_i \tau(p_{ii}) = d_{\pi} \int_B \overline{ \sigma(\gamma,e) }  \sigma( y, y^{-1} e y) V_{\eta} \eta(y^{-1} e y) \dif{y} = d_{\pi}\| \eta \| \int_B \dif{y} = d_{\pi} \vol(G/\Gamma). \]
\end{proof}

\begin{lemma}\label{lem:centralizer_fundom}
Let $B$ be a fundamental domain for $\Gamma$ in $G$ and suppose $\gamma \in \Gamma$ is such that the conjugacy class $C_{\gamma}$ is $\sigma$-regular and finite, say $C_{\gamma} = \{ \beta_1^{-1} \gamma \beta_1 ,\ldots, \beta_k^{-1} \gamma \beta_k \}$. Then
\[ \tilde{B} = \bigcup_{j=1} \beta_j B \]
is a fundamental domain for the $\Gamma_{\gamma}$ (the centalizer of $\gamma$ in $\Gamma$) in $G$.
\end{lemma}

\begin{proof}
First suppose that $\tilde{B} \cap \gamma'\tilde{B} \neq \emptyset$ for some $\gamma' \in \Gamma_{\gamma}$. Then there exist $1 \leq i,j \leq k$ such that $\beta_i B \cap \gamma' \beta_j B \neq \emptyset$. Using that $B$ is a fundamental domain for $\Gamma$ in $G$ this implies that $\beta_i = \gamma' \beta_j$. But then
\[ \beta_i^{-1} \gamma \beta_i = \beta_j^{-1} \gamma'^{-1} \gamma \gamma' \beta_j = \beta_j^{-1} \gamma \beta_j .\]
This forces $i=j$ which gives $B \cap \gamma' B \neq \emptyset$. This can only happen when $\gamma' = e$.

Let $x \in G$. Using that $B$ is a fundamental domain for $\Gamma$ in $G$ we can write $x = \gamma'' y$ where $\gamma'' \in \Gamma$ and $y \in B$. There exists $1 \leq j \leq k$ such that $\gamma''^{-1} \gamma \gamma'' = \beta_j^{-1} \gamma \beta_j$. But then $\gamma'' \beta_j^{-1} \in \Gamma_{\gamma}$ so $\gamma'' = \gamma' \beta_j$ for some $\gamma' \in \Gamma_{\gamma}$. Hence $x = \gamma' (\beta_j y) \in \gamma' \tilde{B}$.
\end{proof}

\begin{proof}[Proof of \Cref{thm:cvfd}]
By \Cref{lem:integral_B} we know that the scalar-valued von Neumann dimension of $_M \Hip$ is equal to $d_{\pi} \vol(G/\Gamma)$. Using the relation between $\tau$ and the center-valued trace in \eqref{eq:cvt_t} we have that
\begin{align*}
    \phi(\gamma) = \tau(\rho_{\overline{\sigma}}(\gamma) \Phi(p)) = \sum_i \tau( \rho_{\overline{\sigma}}(\gamma) \Tr(p_{ii})) = \sum_i \tau( \Tr(\rho_{\overline{\sigma}}(\gamma)) p_{ii} ) .
\end{align*}
Using the formula for the center-valued trace on $N$ \eqref{eq:cvt_r} (and making sure to conjugate the 2-cocycle) as well as \Cref{lem:integral_B} we obtain $\phi(\gamma) = 0$ when $C_{\gamma}$ is infinite or not $\sigma$-regular. When $C_{\gamma}$ is both finite and $\sigma$-regular we obtain
\begin{align*}
    \phi(\gamma) &= |C_{\gamma}|^{-1} \sum_i \sum_{j=1}^k \overline{ \tilde{\sigma}(\gamma, \beta_j)} \tau ( \rho_{\overline{\sigma}}^{\Gamma}(\beta_j^{-1} \gamma \beta_j)  p_{ii}) \\
    &= d_{\pi} |C_{\gamma}|^{-1} \sum_{j=1}^k \overline{ \tilde{\sigma}(\gamma, \beta_j)} \int_B \overline{ \tilde{\sigma}( \beta_j^{-1} \gamma \beta_j, y) } V_{\eta} \eta (y^{-1} \beta_j^{-1} \gamma \beta_j y) \dif{y} \\
    &= d_{\pi} |C_{\gamma}|^{-1} \sum_{j=1}^k \int_B \overline{ \tilde{\sigma}(  \gamma, \beta_j y ) } V_{\eta} \eta (y^{-1} \beta_j^{-1} \gamma \beta_j y) \dif{y} \\
    &= d_{\pi} |C_{\gamma}|^{-1} \sum_{j=1}^k \int_{ \beta_j B} \overline{ \tilde{\sigma}(  \gamma, y ) } V_{\eta} \eta (y^{-1} \gamma y) \dif{y} .
\end{align*}
Hence, using the definition of $\tilde{B}$ from \Cref{lem:centralizer_fundom}, we obtain
\begin{align*}
    \phi(\gamma) &= d_{\pi} |C_{\gamma}|^{-1} \int_{\tilde{B}} \overline{ \tilde{\sigma}(\gamma, y) } V_{\eta} \eta(y^{-1} \gamma y) \dif{y} .
\end{align*}
Note that the function $y \mapsto \tilde{\sigma}(\gamma, \beta_j^{-1}y) V_{\eta} \eta(y^{-1} \gamma y)$ is left $\Gamma_{\gamma}$-invariant. Indeed, if $\gamma' \in \Gamma_{\gamma}$ then $\tilde{\sigma}(\gamma,\gamma') =1$ since $\gamma$ is $\sigma$-regular, so
\begin{align*}
    \overline{ \tilde{\sigma}(\gamma,  \gamma' y) } V_{\eta} \eta(y^{-1} \gamma'^{-1} \gamma \gamma' y) &= \overline{ \tilde{\sigma}(\gamma, \gamma') \sigma( \gamma'^{-1} \gamma \gamma', y) } V_{\eta} \eta(y^{-1} \gamma y) \\
    &= \overline{ \tilde{\sigma}(\gamma,y) } V_{\eta} \eta( y^{-1} \gamma y) .
\end{align*}
Thus, since $\tilde{B}$ is a fundamental domain for $\Gamma_{\gamma}$ in $G$ by \Cref{lem:centralizer_fundom}, we can integrate over $G/\Gamma_{\gamma}$ instead of $\tilde{B}$. This leaves us with the formula in \Cref{thm:cvfd}, finishing the proof.
\end{proof}

For 2-cocycles on abelian groups satisfying Kleppner's condition the center-valued von Neumann dimension takes a particularly simple form:

\begin{corollary}\label{cor:abelian_cdim}
Suppose that $G$ is abelian and that $(G,\sigma)$ satisfies Kleppner's condition. Then the center-valued von Neumann dimension of $_{\vN(\Gamma,\sigma)}\Hip$ is given by
\[ \cdim_{\vN(\Gamma,\sigma)} \Hip = d_{\pi} \vol(G/\Gamma) I .\]
\end{corollary}

\begin{proof}
When $G$ is abelian we have that $C_{\gamma} = \{\gamma \}$ and $\Gamma_{\gamma} = \Gamma$ for every $\gamma \in \Gamma$. Hence the expression in \Cref{thm:cvfd} collapses to
\begin{align*}
    \phi(\gamma) &= d_{\pi} \int_{G/\Gamma} \overline{ \tilde{\sigma}(\gamma,y)} V_{\eta}\eta(y^{-1} \gamma y) \dif{(y\Gamma)} \\
    &= d_{\pi} \langle \eta, \pi(\gamma) \eta \rangle \int_{G/\Gamma} \overline{ \sigma(\gamma, y) }  \sigma(y, \gamma)  \dif{(y \Gamma)} .
\end{align*}
The map $y\Gamma \mapsto \overline{ \sigma(\gamma, y) }  \sigma(y, \gamma) $ is a character on $G/\Gamma$ by \Cref{lem:cocycle_conj}. Since $(G,\sigma)$ is assumed to satisfy Kleppner's condition, this character is trivial if and only if $\gamma = e$. Hence
\[ \phi(\gamma) = d_{\pi} \langle \eta, \pi(\gamma) \eta \rangle \vol(G/\Gamma) \delta_{\gamma,e} = d_{\pi} \vol(G/\Gamma) \delta_{\gamma,e} .\]
Since $\cdim_M \Hip$ is uniquely determined by $\phi$ it follows that
\[ \cdim_{\vN(\Gamma,\sigma)} \Hip = d_{\pi} \vol(G/\Gamma) I .\]
\end{proof}

\section{Applications to frame theory}\label{sec:frame_theory}

\subsection{Frames and Riesz sequences}\label{subsec:frames_riesz_sequences}

Let $\mathcal{H}$ be a (complex) Hilbert space and $J$ an index set. A family $(e_j)_{j \in J}$ in $\mathcal{H}$ is a \emph{frame} for $\mathcal{H}$ if there exist constants $A,B > 0$ such that
\begin{equation*}
    A \| \xi \|^2 \leq \sum_{j \in J} | \langle \xi, e_j \rangle |^2 \leq B \| \xi \|^2 \;\;\; \text{for all $\xi \in \mathcal{H}$.}
\end{equation*}
The numbers $A$ and $B$ are called lower and upper frame bounds, respectively. If one can choose $A=B=1$ in the above equation, the frame $(e_j)_{j \in J}$ is called \emph{Parseval}.

Associated to a frame $(e_j)_{j \in J}$ is the \emph{analysis operator}, which is the injective bounded linear operator $C \colon \Hi \to \ell^2(\Gamma)$ given by
\[ C \xi = (\langle \xi, e_j \rangle )_{j \in J} \;\;\; \text{for $\xi \in \Hi$}. \]
When the frame is Parseval, the analysis operator is an isometry. The \emph{frame operator} is the positive invertible operator $S = C^*C \in \B(\Hi)$ and the associated family $(S^{-1/2} e_j)_j$ is a Parseval frame. Conversely, if $C \colon \Hi \to \ell^2(J)$ is an isometry, then we obtain a Parseval frame $(e_j)_{j \in J}$ in $\Hi$ where $C e_j$ is the orthogonal projection of $\delta_j \in \ell^2(J)$ onto the subspace $C \Hi \subseteq \ell^2(J)$. Thus, the existence of a frame in $\Hi$ indexed by $J$ is equivalent to the existence of an isometry $\Hi \to \ell^2(J)$.

The dual notion to a frame is that of a Riesz sequence. A family $(e_j)_{j \in J}$ is called a \emph{Riesz sequence} for $\mathcal{H}$ if there exist constants $A,B > 0$ such that
\begin{equation*}
    A \| c \|_2^2 \leq \Big\| \sum_{j \in J} c_j e_j \Big\|^2 \leq B \| c \|_2^2 \;\;\; \text{for all $c = (c_j)_j \in \ell^2(J)$.}
\end{equation*}
The numbers $A$ and $B$ are called lower and upper Riesz bounds, respectively. Note that an orthonormal family is precisely a Riesz sequence for which one can choose $A=B=1$. A family $(e_j)_j$ that is both a frame and a Riesz sequence is called a \emph{Riesz basis}.

Associated to a Riesz sequence is the \emph{synthesis operator} $D  \colon \ell^2(J) \to \Hi$ given by
\begin{align*}
     D(c_j)_j = \sum_j c_j e_j \;\;\; \text{for $(c_j)_j \in \ell^2(J)$,}
\end{align*}
which is an injective bounded linear operator. It is isometric when $(e_j)_j$ is orthonormal. A Riesz sequence is always a Riesz basis (in particular a frame) for its closed linear span $\mathcal{K} = \clspn \{ e_j : j \in J \}$, so the restriction $S|_{\mathcal{K}}$ of its frame operator $S$ to $\mathcal{K}$ is invertible. The associated family $(S^{-1/2}e_j)_{j \in J}$ is then orthonormal. Conversely, if $D \colon \ell^2(J) \to \Hi$ is an isometry, then $(D \delta_j)_{j \in J}$ is orthonormal. This shows that the existence of a Riesz sequence in $\Hi$ indexed by $J$ is equivalent to the existence of an isometry $\ell^2(J) \to \Hi$.

\subsection{Multiwindow super systems}\label{subsec:multi_super}

Let $(\pi, \Hip)$ be a $\sigma$-projective unitary representation of a locally compact group $G$ and let $\Gamma$ be a lattice in $G$. We will be interested in frames and Riesz sequences for $\Hip$ of the form
\[ \pi(\Gamma) \eta = ( \pi(\gamma) \eta )_{\gamma \in \Gamma}  \]
for vectors $\eta \in \Hip$. More generally, we define the $n$-multiwindow $d$-super system associated to a matrix $(\eta_{i,j})_{i,j=1}^{n,d}$ of vectors in $\Hip$ to be the $\Gamma \times \{ 1, \ldots, n \}$-indexed family
\begin{equation}
    \Big( ( \pi(\gamma) \eta_{i,1}, \ldots, \pi(\gamma) \eta_{i,d} )  \Big)_{ \gamma \in \Gamma, 1 \leq i \leq n} \label{eq:multi_super}
\end{equation}
in $\Hip^d$. If an $n$-multiwindow $d$-super system is a frame for $\Hip^d$ we call it an \emph{$n$-multiwindow $d$-super frame}. We will say that $(\pi,\Gamma)$ \emph{admits} an $n$-multiwindow $d$-super frame if there exists an $n$-multiwindow $d$-super frame of the form \eqref{eq:multi_super} for some $(\eta_{i,j})_{i,j=1}^{n,d}$, and we call $(\eta_{i,j})_{i,j=1}^{n,d}$ the generators of the frame. We make analogous definitions for Riesz sequences and Riesz bases.

If $d=1$ we obtain the $n$-multiwindow system $( \pi(\gamma) \eta_i  )_{ \gamma \in \Gamma, 1 \leq i \leq n }$ and if $n=1$ we obtain the $d$-super system $\Big( (\pi(\gamma) \eta_1, \ldots, \pi(\gamma) \eta_n) \Big)_{\gamma \in \Gamma }$. If both $n=d=1$ we recover the system $\pi(\Gamma) \eta$.

We need the following representation-theoretic characterizations of the existence of multiwindow super frames and Riesz sequences.

\begin{proposition}\label{prop:frame_riesz_isometries}
The following are equivalent:
\begin{enumerate}
    \item $(\pi,\Gamma)$ admits an $n$-multiwindow $d$-super frame (resp.\ $n$-multiwindow $d$-super Riesz sequence) (resp.\ $n$-multiwindow $d$-super Riesz basis).
    \item $(\pi,\Gamma)$ admits a $n$-multiwindow $d$-super Parseval frame (resp.\ $n$-multiwindow $d$-super orthonormal sequence) (resp.\ $n$-multiwindow $d$-super orthonormal basis).
    \item There exists $\Gamma$-invariant isometry $\Hip^d \to \ell^2(\Gamma)^n$ (resp.\ $\Gamma$-invariant isometry $\ell^2(\Gamma)^n \to \Hip^d$) (resp.\ $\Gamma$-invariant unitary map $\Hip^d \to \ell^2(\Gamma)^n$).
\end{enumerate}
\end{proposition}

\begin{proof}
Let $(\eta_{i,j})_{i,j=1}^{n,d}$ be the generators of an $n$-multiwindow $d$-super frame associated to $(\pi,\Gamma)$. Let $C$ be the associated analysis operator. Then $C\pi(\gamma) = \lambda_{\sigma}(\gamma)C$ for all $\gamma \in \Gamma$. Consequently, the frame operator $S = C^*C$ commutes with $\pi(\gamma)$ for $\gamma \in \Gamma$, so the elements of the associated Parseval frame are of the form
\[  S^{-1/2} ( \pi(\gamma) \eta_{i,1}, \ldots, \pi(\gamma) \eta_{i,d} ) = ( \pi(\gamma) S^{-1/2}\eta_{i,1}, \ldots, \pi(\gamma) S^{-1/2} \eta_{i,d} ) \]
for $\gamma \in \Gamma$ and $1 \leq i \leq n$. In other words, the matrix of vectors $( S^{-1/2} \eta_{i,j} )_{i,j=1}^{n,d}$ generates a $n$-multiwindow $d$-super Parseval frame.

If $(\eta_{i,j})_{i,j=1}^{n,d}$ are the generators of a $n$-multiwindow $d$-super Parseval frame then the coefficient operator $C$ is a $\Gamma$-invariant isometry $\Hip^d \to \ell^2(\Gamma)^n$. Conversely, suppose $C \colon \Hip^d \to \ell^2(\Gamma)^n$ is a $\Gamma$-invariant isometry. Let $e_i$ be the vector $(0, \ldots, \delta_0, \ldots, 0) \in \ell^2(\Gamma)^n$ where $\delta_0$ is in the $i$th position. Then $\{ \lambda_{\sigma}^G(\gamma) e_i : \gamma \in \Gamma, 1 \leq i \leq n \}$ is an orthonormal basis for $\ell^2(\Gamma)^n$. Consequently, if $P$ denotes the projection of $\ell^2(\Gamma)^n$ onto $C(\Hip^d)$, then the vectors $P \lambda_{\sigma}^G(\gamma) e_i = \pi(\gamma) P e_i$ for $\gamma \in \Gamma$ and $1 \leq i \leq n$ form a Parseval frame for $\Hip^d$. Letting $P e_i = (\eta_{i,1} ,\ldots, \eta_{i,d})$, the matrix $(\eta_{i,j})_{i,j=1}^{n,d}$ generates a $n$-multiwindow $d$-super Parseval frame.

The arguments for Riesz sequences are similar. If $(\eta_{i,j})_{i,j=1}^{n,d}$ are the generators of an $n$-multiwindow $d$-super Riesz sequence then the frame operator $S$ restricted to the closed linear span of the Riesz sequence is $\Gamma$-invariant, so the vectors $(S^{-1/2} \eta_{i,j})_{i,j=1}^{n,d}$ are the generators of an $n$-multiwindow $d$-super orthonormal family. The corresponding analysis operator is a $\Gamma$-invariant isometry $\ell^2(\Gamma)^n \to \Hip^d$. Conversely, for a $\Gamma$-invariant isometry $D \colon \ell^2(\Gamma)^n \to \Hip^d$, the vectors $(\eta_{i,j})_{i,j=1}^{n,d}$ are the generators of an $n$-multiwindow $d$-super orthonormal family, where $D \delta_i = (\eta_{i,1}, \ldots, \eta_{i,d})$ for $1 \leq i \leq n$.
\end{proof}

We will also need the following generalization of \cite[Proposition 7.6]{VaRo00}. The strategy of the proof is the same.

\begin{proposition}\label{prop:frame_crit_basis}
Suppose that $(\pi,\Gamma)$ admits an $n$-multiwindow $d$-super frame. If $d_{\pi} \vol(G/\Gamma) = n/d$, then $(\pi,\Gamma)$ admits an $n$-multiwindow $d$-super Riesz basis.
\end{proposition}

\begin{proof}
By \Cref{prop:frame_riesz_isometries}, we can assume that there exist $(\eta_{i,j})_{i,j=1}^{n,d}$ that are generators of an $n$-multiwindow $d$-super Parseval frame. Let $B$ be a fundamental domain for $\Gamma$ in $G$, so that $\{ B \gamma : \gamma \in \Gamma \}$ is a partition of $G$. Then for any $y \in B$ and $(\xi_j)_{j=1}^d \in \Hip^d$, we have that
\begin{align*}
    \sum_{i=1}^n \sum_{\gamma \in \Gamma} | \langle (\xi_j)_{i=1}^n, (\pi(y \gamma) \eta_{i,j})_{j=1}^d \rangle |^2 &= \sum_{i=1}^n \sum_{\gamma \in \Gamma} | \langle (\pi(y)^*\xi_j)_{j=1}^d, (\pi(\gamma) \eta_{i,j})_{j=1}^d \rangle |^2 \\
    &= \| (\pi(y)^* \xi_j)_j \|^2 = \| (\xi_j)_j \|^2 .
\end{align*}
Integrating this equality over $y \in B$ and using the orthogonality relation in \Cref{prop:si_unimodular}, we have that
\begin{align*}
    \| (\xi_j)_j \|^2 \vol(G/\Gamma) &= \int_B \| (\xi_j)_j \|^2 \dif{x} \\
    &= \sum_{i=1}^n \int_B \sum_{\gamma \in \Gamma} | \langle (\xi_j)_j, (\pi(y \gamma) \eta_{i,j})_j \rangle |^2 \dif{x} \\
    &= \sum_{i=1}^n \int_G | \langle (\xi_j)_j, (\pi(x) \eta_{i,j})_j \rangle |^2 \dif{x} \\
    &= \sum_{i=1}^n \int_G \sum_{j,j' = 1}^d \langle \xi_j, \langle \pi(x) \eta_{i,j'} \rangle \overline{ \langle \xi_{j'}, \pi(x) \eta_{i,j} \rangle } \dif{x} \\
    &= d_{\pi}^{-1} \sum_{i=1}^n \sum_{j,j'=1}^d \langle \xi_j, \xi_{j'} \rangle \overline{ \langle \pi(x) \eta_{i,j'}, \pi(x) \eta_{i,j} \rangle } .
\end{align*}
Picking $\xi_1, \ldots, \xi_n$ so that $\langle \xi_j, \xi_{j'} \rangle = \delta_{j,j'}$ for $1 \leq j,j' \leq d$, we get
\[ d \vol(G/\Gamma) = \| (\xi_j)_j \|^2 \vol(G/\Gamma) = d_{\pi}^{-1} \sum_{i=1}^n \sum_{j=1}^d \| \pi(x) \eta_{i,j} \|^2 = d_{\pi}^{-1} \sum_{i=1}^n \sum_{j=1}^d \| \eta_{i,j} \|^2 . \]
Since we assume that $d_{\pi} \vol(G/\Gamma) = n/d$, we get
\begin{equation}
    \sum_{i=1}^n \sum_{j=1}^d \| \eta_{i,j} \|^2 = n . \label{eq:proof_dens}
\end{equation}
For each $1 \leq i \leq n$, the vector $(\eta_{i,j})_{j=1}^d \in \Hip^d$ is a member of a Parseval frame, hence $\| (\eta_{i,j})_j \|^2 = \sum_{j=1}^d \| \eta_{i,j} \|^2 \leq 1$. Combining this with \eqref{eq:proof_dens}, we must have $\| (\eta_{i,j})_j \| = 1$ for each $1 \leq i \leq n$. But then every vector in the Parseval frame generated by $(\eta_{i,j})_{i,j=1}^{n,d}$ is a unit vector, so the Parseval frame must be an orthonormal basis.
\end{proof}

\subsection{The density theorem and converses}

Using the results on the center-valued von Neumann dimension of $\Hip$ as a Hilbert module over the $\sigma$-twisted group von Neumann algebra $\vN(\Gamma,\sigma)$, we can now characterize the existence of $n$-multiwindow $d$-super frames in terms of the function $\phi$ from \Cref{thm:cvfd}.

\begin{theorem}\label{thm:dens_conv}
Let $G$ be a second-countable, unimodular, locally compact group, let $\sigma$ be a 2-cocycle on $G$, and let $(\pi, \Hip)$ be a $\sigma$-projective, irreducible, square-integrable, unitary representation of $G$. Let $\Gamma$ be a lattice in $G$. Let $\phi$ be as in \Cref{thm:cvfd}. Then the following hold:
\begin{enumerate}
    \item $(\pi,\Gamma)$ admits an $n$-multiwindow $d$-super frame if and only if $ (n/d)\delta_e - \phi$ is a $\sigma$-positive definite function.
    \item $(\pi,\Gamma)$ admits an $n$-multiwindow $d$-super Riesz sequence if and only if $ \phi - (n/d) \delta_e$ is a $\sigma$-positive definite function.
    \item $(\pi,\Gamma)$ admits an $n$-multiwindow $d$-super Riesz basis if and only if $\phi = (n/d) \delta_e$.
\end{enumerate}
\end{theorem}

\begin{proof}
Set $M = \vN(\Gamma,\sigma)$. By \Cref{prop:frame_riesz_isometries}, $(\pi,\Gamma)$ admits an $n$-multiwindow $d$-super frame if and only if there is an $\pi|_{\Gamma}$-invariant isometry $\Hip^d \to \ell^2(\Gamma)^n$. This is the case if and only if $_{M} \Hip^d$ is a submodule of $_{M} \ell^2(\Gamma)^n$. By \Cref{prop:cdim_iso} this is the case if and only if
\[ d \cdim_M \Hip = \cdim_{M} \Hip^d \leq \cdim_M \ell^2(\Gamma)^n = n I .\]
By \Cref{thm:cvfd} the center-valued von Neumann dimension $T \coloneqq \cdim_M \Hip$ is given by convolution with the function $\phi$, hence determined by the values $\phi(\gamma) = \tau( \rho_{\overline{\sigma}}(\gamma) T)$ for $\gamma \in \Gamma$. The condition $T \leq (n/d)I$ is equivalent to $\langle ((n/d)\delta_e - \phi)*_{\sigma} f, f \rangle = \langle ((n/d)I-T)f, f \rangle \geq 0$ for all finitely supported functions $f$ on $\Gamma$. By \Cref{prop:pos_def} this happens exactly when $(n/d)\delta_e - \phi$ is $\sigma$-positive definite.

Using \Cref{prop:frame_riesz_isometries} in a similar manner shows that $(\pi,\Gamma)$ admits an $n$-multiwindow $d$-super Riesz sequence if and only if $\phi - (n/d)\delta_e$ is $\sigma$-positive definite. Combining the statements for frames and Riesz sequences, we see that if $(\pi,\Gamma)$ admits an $n$-multiwindow $d$-super Riesz basis, then $\phi = (n/d) \delta_e$. Conversely, suppose $\phi = (n/d) \delta_e$. Then in particular, $d_{\pi} \vol(G/\Gamma) = n/d$. Since $\phi$ is $\sigma$-positive definite, $(\pi,\Gamma)$ admits an $n$-multiwindow $d$-super frame by what we already proved. By \Cref{prop:frame_crit_basis}, $(\pi,\Gamma)$ then admits an $n$-multiwindow $d$-super Riesz basis.
\end{proof}

As a corollary we get the following generalization of the density theorem from \cite{VaRo00} to $n$-multiwindow $d$-super systems:

\begin{theorem}
Let $G$ be a second-countable, unimodular, locally compact group, let $\sigma$ be a 2-cocycle on $G$, and let $(\pi, \Hip)$ be a $\sigma$-projective, irreducible, square-integrable, unitary representation of $G$. Let $\Gamma$ be a lattice in $G$. Then the following hold:
\begin{enumerate}
    \item If $(\pi,\Gamma)$ admits an $n$-multiwindow $d$-super frame, then $d_{\pi} \vol(G/\Gamma) \leq n/d$.
    \item If $(\pi,\Gamma)$ admits an $n$-multiwindow $d$-super Riesz sequence, then $d_{\pi} \vol(G/\Gamma) \geq n/d$.
    \item If $(\pi,\Gamma)$ admits an $n$-multiwindow $d$-super Riesz basis, then $d_{\pi} \vol(G/\Gamma) = n/d$.
\end{enumerate}
\end{theorem}

\begin{proof}
Follows from \Cref{thm:dens_conv} and the fact that $d_{\pi} \vol(G/\Gamma) = \dim_M \Hip = \tau(\cdim_M \Hip)$.
\end{proof}

Combining \Cref{cor:abelian_cdim} and \Cref{thm:dens_conv} we get another immediate corollary, which gives a complete converse to the density theorem when $G$ is abelian and $(G,\sigma)$ satisfies Kleppner's condition.

\begin{theorem}\label{thm:dens_conv_abelian}
Let $G$ be a second-countable, abelian, locally compact group, let $\sigma$ be a 2-cocycle on $G$, and let $(\pi, \Hip)$ be a $\sigma$-projective, irreducible, square-integrable, unitary representation of $G$. Suppose that $(G,\sigma)$ satisfies Kleppner's condition. Let $\Gamma$ be a lattice in $G$. Then the following hold:
\begin{enumerate}
    \item $(\pi,\Gamma)$ admits an $n$-multiwindow $d$-super frame if and only if $d_{\pi} \vol(G/\Gamma) \leq n/d$.
    \item $(\pi,\Gamma)$ admits an $n$-multiwindow $d$-super Riesz sequence if and only if $d_{\pi} \vol(G/\Gamma) \geq n/d$.
    \item $(\pi,\Gamma)$ admits an $n$-multiwindow $d$-super Riesz basis if and only if $d_{\pi} \vol(G/\Gamma) = n/d$.
\end{enumerate}
\end{theorem}

\subsection{Gabor analysis}\label{subsec:gabor}

We end with an application to Gabor analysis on locally compact abelian groups \cite{Gr98}. Let $A$ be a second-countable, locally compact abelian group with Pontryagin dual $\widehat{A}$ and set $G = A \times \widehat{A}$. The \emph{Weyl--Heisenberg 2-cocycle} of $G$ is given by
\[ \sigma( (x,\omega), (x',\omega')) = \overline{ \omega'(x) } \;\;\; \text{for $(x,\omega),(x',\omega') \in G$.} \]
Note that $(G,\sigma)$ satisfies Kleppner's condition: Indeed, suppose $(x,\omega) \in G$ is such that $\sigma((x,\omega),(x',\omega')) = \sigma((x',\omega'),(x,\omega))$ for all $(x',\omega') \in G$. Then $\omega'(x) = \omega(x')$ for all $x' \in A$ and $\omega' \in \widehat{A}$. Setting $\omega' = 1$ gives that $\omega$ is the trivial character and setting $x' = 1$ gives that $\omega'(x) = 1$ for all $\omega' \in \widehat{A}$ which implies that $x = e$ by Pontryagin duality.

The \emph{Weyl--Heisenberg representation} is the square-integrable, irreducible, $\sigma$-projective representation of $G$ on $L^2(A)$ given by
\[ \pi(x,\omega) \xi(t) = \omega(t) \xi(x^{-1}t) \;\;\; \text{for $(x,\omega) \in G$, $\xi \in L^2(A)$ and $t \in A$.} \]
The orthogonality relations for the short-time Fourier transform (\cite{Gr98}) yield that $d_{\pi} = 1$. In this setting a system of the form $\pi(\Gamma) \eta$ for some $\eta \in L^2(A)$ and $\Gamma$ a lattice in $G = A \times \widehat{A}$ is called a \emph{Gabor system}. If it has the frame property in $L^2(A)$ we call it a \emph{Gabor frame}, and similarly for Riesz sequences and Riesz bases. We also speak of $n$-multiwindow $d$-super Gabor frames and Gabor Riesz bases where the definitions are according to \Cref{subsec:multi_super}. The following theorem is an immediate consequence of \Cref{thm:dens_conv_abelian}.

\begin{theorem}\label{thm:gabor_existence}
Let $A$ be a second-countable, locally compact abelian group, and let $\Gamma$ be a lattice in $A \times \widehat{A}$. Then the following hold:
\begin{enumerate}
    \item There exists an $n$-multiwindow $d$-super Gabor frame over $\Gamma$ if and only if $\vol(G/\Gamma) \leq n/d$.
    \item There exists an $n$-multiwindow $d$-super Gabor Riesz sequence over $\Gamma$ if and only if $\vol(G/\Gamma) \geq n/d$.
    \item There exists an $n$-multiwindow $d$-super Gabor Riesz basis over $\Gamma$ if and only if $\vol(G/\Gamma) = n/d$.
\end{enumerate}
\end{theorem}

The above theorem applies e.g.\ to the case where $A$ is the adele group of a global field which was studied in \cite{EnJaLu19,EnJaLu20}.

\printbibliography

\end{document}